\documentclass[11pt,draft, reqno]{amsart}

\usepackage[margin=22mm]{geometry}
\usepackage{amssymb, amsmath, amsthm, amscd}

\usepackage[T1]{fontenc}
\usepackage{eucal,mathrsfs,dsfont}
\usepackage{color}

\renewcommand{\leq}{\leqslant}
\renewcommand{\geq}{\geqslant}
\renewcommand{\le}{\leqslant}
\renewcommand{\ge}{\geqslant}

\def\eps{\varepsilon}

\definecolor{mno}{rgb}{0.5,0.1,0.5}

\newcommand{\R}{\mathds R}
\newcommand{\bS}{\mathds S}

\newcommand{\e}{\varepsilon}

\newcommand{\Pp}{\mathds P}
\newcommand{\Ee}{\mathds E}

\newcommand{\I}{\mathds 1}
\newcommand{\w}{\omega}

\newcommand{\Z}{\mathds Z}

\newcommand{\D}{\mathscr{D}}

\newcommand{\F}{\mathscr{F}}

\newcommand{\LL}{\mathcal{L}}

\newtheorem{theorem}{Theorem}[section]
\newtheorem{lemma}[theorem]{Lemma}
\newtheorem{proposition}[theorem]{Proposition}

\theoremstyle{definition}

\newtheorem{remark}[theorem]{Remark}

\numberwithin{equation}{section}

\begin{document}
\allowdisplaybreaks

\title[long range random
walks in  balanced random environments] {\bfseries Quenched
invariance principle for
long range random walks in  balanced
random environments}
\author{Xin Chen\qquad Zhen-Qing Chen\qquad Takashi Kumagai\qquad Jian Wang}

\maketitle

\begin{abstract}
We establish via a probabilistic approach the
quenched invariance principle for a class of long range
random walks in independent (but not necessarily identically distributed) balanced random environments,
with the transition probability from $x$ to $y$  on average being comparable to
$|x-y|^{-(d+\alpha)}$ with $\alpha\in (0,2]$.
We use the martingale property to estimate exit time from balls
and establish
tightness of the scaled processes, and apply the uniqueness of the martingale problem
to identify the limiting process.
When $\alpha\in (0,1)$, our approach works even for non-balanced cases.
When $\alpha=2$, under a diffusive with the logarithmic perturbation scaling,
we show that the limit of scaled processes is a
Brownian motion.

\bigskip


\medskip

\end{abstract}
\allowdisplaybreaks

\section{Introduction and main result}
Let $(\Omega, \F, \Pp)$ be a probability space. We will use $\Ee$ to denote the mathematical expectation
taken under $\Pp$.
Let  $\Z$ and $\Z_+$ denote the set of integers and non-negative integers, respectively. For $d\ge 1$,
let $\Z^d$ be the $d$-dimensional integer lattice,
and $\Z_{\bf 0}^d:=\Z^d\setminus\{ {\bf0}\}$, where
${\bf 0}:=(0, \cdots, 0)$ is the zero element in  $\Z^d$ (or $\R^d)$. Here and in what follows, we use $:=$ as way of definition.
For each fixed $\w\in \Omega$, we consider a continuous time random walk  $X^\w:=(X_t^\w)_{t\ge 0}$
on $\Z^d$ with the  infinitesimal
generator
$$
\LL^{\w}f(x):=\sum_{z\in
\Z_{\bf 0}^d}\big(f(x+z)-f(x)\big)\frac{\kappa(x,z)(\w)}{|z|^{d+\alpha}},\quad
f\in B_b(\Z^d),
$$
where $0<\alpha\le 2$ and $\kappa(x,z)(\w)$ satisfies the following balanced condition
\begin{equation}\label{balacd}
\kappa(x,z)(\w)=\kappa(x,-z)(\w)\ge 0
\quad \hbox{ for all }
x\in \Z^d, z\in \Z_{\bf 0}^d \hbox{ and } \w \in \Omega.
\end{equation}
In other words, $X^\w$ is a continuous time Markov process on $\Z^d$
such that conditioned on $X^\w_t=x$, $X^\w$ waits an exponentially distributed random
amount of time with parameter
$$
\lambda^\w(x):=\sum_{z\in \Z_{\bf 0}^d}\kappa(x, z)(\w)|z|^{-(d+\alpha)}
$$
before it jumps to $x+z$ with probability $\kappa(x, z)(\w)|z|^{-(d+\alpha)}/\lambda^\w (x)$.
When
$\alpha\in (0,2)$ and $\kappa(x,z)(\w)$ is uniformly elliptic, i.e., there exist
 (non-random)
 constants $0<c_1\le c_2<\infty$ such that $c_1\le \kappa(x,z)(\w)\le c_2$ for all
 $x\in \Z^d,$ $z\in \Z_{\bf 0}^d$, and $\w\in \Omega$, the process $X^\w$ is called an
$\alpha$-stable-like (balanced but not
necessarily symmetric)
random walk in the literature.

If $\kappa(x,z)(\w)=0$ for all $\w\in \Omega$, $x\in \Z^d$ and $z\in \Z^d$ with $|z|>1$, then the process
$X^\w$ is reduced into a nearest neighbor random walk
in balanced  random environments
(NNBRW). In particular, a NNBRW can
be viewed as the discrete counterpart corresponding to an $\R^d$-valued diffusion process in
balanced random
environments, which was
initially considered by Papanicolaou and Varadhan \cite{PV}. The NNBRW was first introduced by Lawler \cite{Law}, where the quenched invariance principle
was established under some uniformly elliptic condition on the conductance in ergodic environments. Thirty years later, Guo and Zeitouni \cite{GZ} proved the
quenched invariance principle for NNBRWs
under some inverse moment condition on the conductance in balanced random environments.
When the environment is i.i.d.,
(that is, $\{\kappa(x,\cdot)\}_{x\in \Z^d}$ are i.i.d. across the sites $x\in \Z^d$ under $\Pp$,)
the quenched invariance principle for NNBRWs was established by Berger and Deuschel \cite{BD} under
the so-called \lq\lq genuinely $d$-dimensional\rq\rq\, condition, where the
 environment is allowed to be degenerate. More recently, under such \lq\lq genuinely $d$-dimensional\rq\rq\, condition, Berger, Cohen, Deuschel and Guo \cite{BCDG} obtained
some quantitative estimates for the solution to elliptic equations, as well as the elliptic
Harnack inequality. For the
NNBRW in time-dependent
balanced environments, the quenched invariance principle has been proven by Deuschel, Guo and
Ramirez in \cite{DGR}, while the quenched local central limit theorem was established by
Deuschel and Guo in \cite{DG}.

\medskip

The goal of this paper is to establish
the quenched
invariance principle for long range random walks in balanced  random environments, i.e.,
$\kappa(x,z)(\w)$ does not vanish when $|z|>1$.
The following assumption will be in force in the paper.

\medskip

\noindent {\bf Assumption (A0)}{ \it
$\{\kappa(x,z): x\in \Z^d, z\in \Z_{+,*}^d
\}$
is a sequence of independent non-negative
random variables, where $$\Z_{+,*}^d:=\big\{x=(x_1,\cdots, x_d)\in \Z^d: x_{i_0}>0\ {\rm for}\
i_0:=\min\{1\le i\le d: x_i\neq 0\}\big\}.$$}

Note that $\Z_{\bf 0}^d=  \Z_{+,*}^d \cup (- \Z_{+,*}^d )
$, and $\kappa (x, z)$
satisfies the balanced condition \eqref{balacd} (i.e. $\kappa(x,z)$ is symmetric in $z\in\Z_{\bf 0}^d$ for each fixed $x\in \Z^d$).
So $\kappa (x, z)$ on $\Z^d\times \Z_{\bf 0}^d$ is determined by its values on $\Z^d \times \Z_{+,*}^d $.

\medskip

We will consider the following assumption when $\alpha\in (0,2)$.

\medskip

\noindent {\bf Assumption (A1)} {\it
Assume $\alpha\in (0,2)$
and $d>4-2\alpha$.

\begin{itemize}
\item [(i)]
There is some constant
$p>\max \left\{\frac{2(d+1)}{d},
\frac{d+1}{2-\alpha} \right\}$ so that
\begin{equation}\label{e:1.2}
\sup_{x\in \Z^d, z\in\Z^d_{\bf 0} }\Ee \left[
\kappa(x,z)^{p} \right]<\infty.
\end{equation}

\item [(ii)] There exists a bounded
continuous
function $K:\R^d\times \R_{\bf 0}^d\rightarrow \R_+:=[0,\infty)$ such that
for  every  large
integer $R\geq 1$
and small $\varepsilon\in (0,1)$,
\begin{equation}\label{e:1.3}
\lim_{n \rightarrow \infty}\sup_{x\in \Z^d,z\in \Z_{\bf0}^d:  \atop{|x|\le nR, \varepsilon n<|z|<n/\varepsilon}} \big|\Ee [\kappa(x, z)]-K(x/n,z/n)\big|=0,
\end{equation}
and that the solution to the martingale problem for
$(\bar \LL , C_c^2 (\R^d))$
is unique,
 where
  for $f\in C^\infty_c(\R^d)$
\begin{equation}\label{l5-3-1a}\begin{split}
\bar \LL f(x):&= {\rm p.v.} \int_{\R_{\bf 0}^d}
\big(f(x+z)-f(x)\big)\frac{K(x,z)}{|z|^{d+\alpha}}\,dz \\
&=  \int_{\R_{\bf 0}^d}
\big(f(x+z)-f(x)-\nabla f(x)\cdot z \I_{\{|z|\leq 1\}} \big)\frac{K(x,z)}{|z|^{d+\alpha}}\,dz,
 \end{split}\end{equation} $\R^d_{\bf 0}:=\R^d\setminus\{\bf0\}$,
and {\rm p.v.} stands for principal value.
\end{itemize}}

\medskip

\begin{remark} \label{R:1.1} \rm
\begin{enumerate}
\item [(i)] Condition \eqref{e:1.3} along with the continuity of $K(x,z)$ on $\R^d\times \R^d_{\bf0}$ implies that
\begin{equation}\label{e:1.5}
K(\lambda x, \lambda z) = K(x, z) \quad \hbox{for every } x, z \in \R^d \hbox{ and } \lambda >0,
\end{equation}
and if we write $z= r\theta \in \R^d_{\bf0}$  in spherical coordinates    with $r=|z|$ and $\theta = z/|z|\in \bS^{d-1}$,
then for every $x\in \R^d$,
$$
\lim_{r\to \infty} K(x,  r \theta ) = \lim_{r\to \infty} K(x/r,   \theta ) = K(0,  \theta ).
$$
In particular, \eqref{e:1.5} says that $K(x, z)$ is uniquely determined by its value on the unit sphere in $\R^{2d}$
and hence $K(x, z)$ is bounded on $\R^d\times \R^d$.

\item[(ii)]
  Observe that conditions \eqref{balacd} and \eqref{e:1.3}
along with the continuity of $K(x,z)$ on $\R^d\times \R^d_{\bf0}$
imply that $K(x, z)$ is balanced in the sense that
\begin{equation}\label{K-ban}
K(x, z)=K(x, -z) \quad \hbox{for all } x\in\R^d, z\in \R_{\bf 0}^d.
\end{equation}
Moreover,
 by \eqref{e:1.2}, \eqref{e:1.3} and the continuity of $K(x,z)$ on $\R^d\times \R^d_{\bf0}$,
$$\sup_{x\in \R^d, z\in \R^d_{\bf0}} K(x, z)\leq \sup_{x\in \Z^d, z\in \Z_{\bf0}^d} \Ee \kappa (x, z).$$
Note also that here we only assume the uniqueness of the martingale problem for $(\bar \LL, C_c^2(\R^d))$,
since the boundedness and the continuity of $K(x,z)$ on $\R^d\times \R^d_{\bf0}$ ensure that the existence of the martingale problem for $(\bar \LL, C_c^2(\R^d))$; see \cite[Theorem 4.1]{BT}.
Some sufficient conditions for the uniqueness of the solution to the martingale problem for $\bar \LL $ defined by \eqref{l5-3-1a} are known; for example, when
$K(\cdot,\cdot)$ satisfies
$0<K_1\le K(x,z)\le K_2<\infty$ for all $x\in \R^d$ and $z\in \R_{\bf0}^d$, and
$\int_0^1 r^{-1}\psi_K(r)\,dr<\infty$, where
$$\psi_K(r)=\sup_{z\in \R_{\bf0}^d, x,y\in \R^d \mathrm{\,\, with\,\, }|x-y|\le r}|K(x,z)-K(y,z)|,$$
see \cite[Theorem 1.2]{BT}, \cite[Theorem 4.6]{CZ}
and \cite[Theorem 1.3]{P}  for related work.
\end{enumerate}
\end{remark}

\medskip

We need the following assumption instead of Assumption {\bf(A1)} when $\alpha=2$.

\medskip

 \noindent {\bf Assumption (A2)} {\it
Assume $\alpha=2$
and $d\geq 1$.
\begin{itemize}
\item[(i)]
There are some constants $c_*>0$ and
$\eta \in (1, 2)$
such that
\begin{equation}\label{l5-3-1b}
\sup_{x\in \Z^d, z\in \Z^d_{\bf0}}\Ee
\left[ \exp\left(  c_* \kappa(x,z)^\eta\right)  \right]<\infty .
\end{equation}

\item[(ii)] There is a constant matrix $A:=(a_{ij})_{1\le i,j\le d}$ so that
for every $R>1$
and $1\leq i, j\leq d$,
\begin{equation}\label{matix}
\lim_{n \rightarrow \infty}\sup_{x\in n^{-1}\Z^d: |x|\le R}\bigg|\frac{1}{\log (1+n)} \sum_{z\in
\Z_{\bf 0}^d: |z|\le n} \frac{z_iz_j\Ee[\kappa(nx,z)]}
{|z|^{d+2}}-a_{ij}\bigg|=0.
\end{equation}
\end{itemize}}

\medskip
\ \

{\rm
For every $n\ge 1$ and $t>0$, let
$$X^{\w,n}_t:=\begin{cases}n^{-1}X^\w_{n^\alpha t},&\quad \alpha\in (0,2),\\
{n}^{-1}X^\w_{{n^2}t/{\log (n+1)} },
&\quad \alpha=2.\end{cases}$$
For any $T>0$, denote by $\Pp_x^{\w}$ (for simplicity we omit $T$ in the notation) the distribution of
$X^\w:=(X_t^\w)_{0\le t\le T}$ on the Skorohod space $\D([0,T];\R^d)$ with initial starting
point $x\in \Z^d$, and denote by
$\Pp_x^{\w,n}$ the distribution of $X^{\w,n}:=(X^{\w,n}_t)_{0\le
t\le T}$ on $\D ([0,T];\R^d)$ starting
at $x\in n^{-1}\Z^d$.}

\begin{theorem}\label{t5-2}
\begin{itemize}
\item[(i)]
Suppose that $\alpha\in (0,2)$, and that Assumptions {\bf (A0)} and {\bf (A1)} hold.  Then for every $T>0$
and a.s. $\w\in \Omega$, $\Pp^{\w,n}_{\bf0}$ converges weakly to $\bar
\Pp_{\bf0}$ on $\D([0,T];\R^d)$, where $\bar \Pp_{\bf0}$ is the
distribution of the unique solution to the martingale problem
for $(\bar {\LL}, C^2_c(\R^d))$ defined by \eqref{l5-3-1a}
starting at ${\bf0}$.

\item[(ii)]
Suppose that $\alpha=2$, and that Assumptions {\bf (A0)} and {\bf (A2)} hold.  Then for every $T>0$
and a.s. $\w\in \Omega$, $\Pp^{\w,n}_{\bf0}$ converges weakly to $\bar
\Pp_{\bf0}$ on $\D([0,T];\R^d)$, where $\bar \Pp_{\bf0}$ is the
distribution of
Brownian motion
starting at ${\bf0}$  with a deterministic
non-negative definite covariance matrix $A$.
\end{itemize}
\end{theorem}

\medskip

We have the following sufficient condition for
\eqref{matix} in Assumption {\bf (A2)}(ii).

\begin{proposition}\label{t5-4}
Suppose  that $\alpha=2$, $d\ge 1$ and $\sup_{x\in \Z^d, z\in\Z^d_{\bf 0} }\Ee \left[
\kappa(x,z)\right]<\infty$.
If  there exists
a bounded and continuous
function $K:\R_{\bf 0}^d \rightarrow \R_+$ such that
 for every
 integer $R\geq 1$
 and  $\varepsilon\in (0,1)$,
\begin{equation}\label{t5-4-1}
\lim_{n\to \infty}\sup_{x\in n^{-1}\Z^d, z\in  n^{-1}\Z_{\bf0}^d
\atop{|x|\le nR,
n^\varepsilon\le |z|\le n}} \big|\Ee [\kappa(x, z)]-K(z/n)\big|=0,
\end{equation}
then
\eqref{matix} holds with
$$
a_{ij}=\int_{\mathbb{S}^{d-1}}K(\theta) \theta_i \theta_j \, d\theta,
$$
where $\mathbb
{S}^{d-1}:=\{z\in \R^d: |z|=1\}$ and $d\theta$ is the standard Lebesgue surface measure on $\mathbb{S}^{d-1}$.
\end{proposition}

\begin{remark}\label{R:1.4}
\begin{itemize}
\item[(i)] Just as that in Remark \ref{R:1.1},  condition \eqref{t5-4-1} together with the continuity of the function $K(z)$ on
$\R_{\bf 0}^d$ implies that $K(z)$ is a homogenous function of degree 0; that is,
$$
K( z) = K(z/|z|) \quad \hbox{for every } z \in \R_{\bf 0}^d.
$$

\item[(ii)]
To the best of our knowledge, all the existing literature including
references \cite{BCDG,BD,DG,DGR,GZ,Law} quoted above are concerned
with nearest neighbor random walks in balanced random environments. This is the first paper to investigate the quenched invariance principle for long range random walks in balanced random environments.
The reader is referred
to \cite{CKW} for the corresponding work on random conductance models with (symmetric) $\alpha$-stable-like jumps. Compared with assumptions in \cite{BCDG,BD,DG,DGR,GZ,Law},
our method can be applied
to non-ergodic environments since $\{\kappa(x,z)(\w): x\in \Z^d,z\in \Z^d_{+,*}, \w\in \Omega\}$
are not
required to be identically distributed.

 \item [(iii)]
The essential character of long range random walk in the present paper is that, the probability of jump from $x$ to $y$ is of the order $|x-y|^{-d-\alpha}$ with some $\alpha\in (0,2]$.
In particular, this indicates that the second moment of the process $X$ is infinite. Therefore,
we can not expect the scaling process to be a Brownian motion, or can not take the diffusive scaling to study the invariance principle. Note that, when $\alpha\in (0,2)$, $|x-y|^{-d-\alpha}$ is a typical transition density for $\alpha$-stable random walk on $\Z^d$. So, in this case it is natural to adopt the $\alpha$-stable scaling and expect the limit process to be a stable-like process.
Also due to this observation, the average of the sum for $\kappa(x,z)$ in large scale, which converges
to $\Ee[\kappa(x,z)]$ by the Borel-Cantelli arguments under the independent property of $\kappa(x,z)$ and some moments conditions on $\kappa(x,z)$, will be directly involved into the jumping kernel of limit process.
The case that $\alpha=2$ is distinct from $\alpha\in (0,2)$.
However,
by the similar method for $\alpha\in (0,2)$,
it holds that
with proper scaling (which is not the diffusive scaling) the limit process is a Brownian motion, whose coefficients are determined by $\Ee[\kappa(x,z)]$ as well. We should emphasize that the framework for $\alpha>2$ is completely different. Roughly speaking, when $\alpha>2$ the second moment of the process $X$ is finite. Even it is believed that by taking the diffusion scaling the limit process should be a
Brownian motion, we do not know how to prove it in general.
See Subsection \ref{e:diff} for more remarks.

\item[(iv)]
In order to establish the quenched invariance principle and
related results
for NNBRWs in either ergodic environments or i.i.d.\ environments,
some analytic tools, in particular the so-called Aleksandrov-Bakelman-Pucci (ABP) type estimates, play a crucial role; see \cite{BCDG,BD,DG,DGR,GZ,Law}. However, for the non-symmetric $\alpha$-stable-like
operator $\bar \LL$ defined by \eqref{l5-3-1a}, the ABP type estimates are still unknown except for a very special class of coefficients
$K(x,z)$, see \cite{GS} for more details.
In this paper, we
use a
probabilistic approach to tackle the quenched invariant principle. This approach
is completely different from those
in the above mentioned papers. We believe that our paper
provides another reasonable approach to study
quenched invariance principle for random walks in balanced random environments.
In particular, our approach
can efficiently handle the $\alpha=2$ case,
  which is of interest in its own.

\item[(v)]
  We emphasize that, unlike random conductance models as considered in \cite[Theorem 1.1]{CKW}, in this paper we do not require the inverse moment condition such as
$\sup_{x\in \Z^d, z\in \Z_{\bf0}^d}\Ee \left[  \kappa(x,z)^{-q}\right] <\infty$ for
some $q>0$. Some kind of non-degenerate condition is partly implied by the
uniqueness assumption
of the martingale problem for
$(\bar \LL, C^2_c(\R^d))$ of \eqref{l5-3-1a}.
For instance, all the literature we know concerning the uniqueness of the martingale problem for the non-local operator $\bar \LL$ defined by \eqref{l5-3-1a} are proved under some kind of non-degeneracy of the associated coefficients.  In particular, the uniqueness of the martingale problem for the  operator $\bar \LL$ defined by \eqref{l5-3-1a} along with \eqref{e:1.3}
could indicate non-degeneracy of the function   $x\mapsto \lim_{z\in\Z^d_{{\bf0}}\,\, \mathrm{ with }\,\, |z|\rightarrow \infty}\Ee[\kappa(x,z)]$.
Moreover, we do not need to assume the balanced condition \eqref{balacd}
for $0<\alpha<1$; see
Subsection \ref{subsecnonba} below.

\item[(vi)]
By \eqref{e:1.3} and the continuity of $K(x,z)$, it is easy to verify that
$K(x,z)=K(sx,sz)$
for every $x\in \R^d$, $z\in \R^d_{\bf0}$  and $s>0$,
which imply that the process $(Y_t)_{t\ge 0}$
whose distribution is the unique solution to the martingale problem for $(\bar {\LL}, C^2_c(\R^d))$ on $\D([0,\infty);\R^d)$ as in Theorem \ref{t5-2} (i)
satisfies the following scaling property
$$
(\lambda Y_{\lambda^{-\alpha}t})_{t\ge 0}\overset{d}=(Y_t)_{t\ge 0},\quad \lambda >0.
$$
Here  $\overset{d}=$  means the distribution of two processes are the same. Moreover,
the limiting operator $\bar \LL$ defined by \eqref{l5-3-1a}
may not be symmetric with respect to the Lebesgue measure, which is different from the
behavior of NNBRWs studied in \cite{BCDG,BD,DG,DGR,GZ,Law}. This non-symmetry nature of the limit operator $\bar \LL$ is due to the fact
that $\{\kappa(x,z)(\w): x\in \Z^d, z\in \Z^d_{+,*},\w\in \Omega\}$ may not be identically distributed.

\item[(vii)]
 By the proof below, we
see
 that the conclusion of Theorem \ref{t5-2} (ii)
holds true if the condition \eqref{matix} is replaced by  the following
\begin{equation}\label{r1-1}
\begin{split}
\lim_{n \rightarrow \infty}\sup_{x\in n^{-1}\Z^d: |x|\le R}\bigg|\frac{1}{\log (1+n)} \sum_{z\in \Z_{\bf0}^d: |z|\le n} \frac{z_iz_j\Ee[\kappa(nx,z)]}
{|z|^{d+2}}-a_{ij}(x)\bigg|=0
\end{split}
\end{equation} for all
integer $R\ge 1$ and $1\le i,j\le d$,
where $a_{ij}(x):\R^d\rightarrow \R$ is a bounded and continuous function. It further follows from \eqref{r1-1} and the continuity of $a_{ij}$ that
$$
a_{ij}(sx)=a_{ij}(x)  \quad \hbox{for every } x\in \R^d \hbox{ and }  s>0.
$$
In particular, for any $x\in \R^d$, $a_{ij}(x)=\lim_{s\to 0} a_{ij}(sx)=a_{ij}({\bf0})$; that is, $a_{i,j}(x)$ is a constant function.
Hence, \eqref{r1-1} is essentially equivalent to \eqref{matix}.

\item [(viii)]
Note that by \eqref{matix}, for every $\xi=(\xi_1,\cdots,\xi_d)\in \R^d$ and
$x\in \Z^d$, it holds that
\begin{equation}\label{e1-4}
\begin{split}
\sum_{i,j=1}^d
  a_{ij}  \xi_i\xi_j&=\lim_{n\to \infty}
\frac{1}{\log (1+n)} \sum_{z\in \Z_{\bf 0}^d: |z|\le n} \sum_{i,j=1}^d\frac{\xi_i \xi_j z_iz_j\Ee[\kappa(x,z)]}
{|z|^{d+2}}\\
&=\lim_{n\to \infty}\frac{1}{\log (1+n)}\sum_{z\in \Z_{\bf 0}^d: |z|\le n}\left(\sum_{i=1}^d \xi_iz_i\right)^2\frac{\Ee[\kappa(x,z)]}
{|z|^{d+2}}\ge 0.
\end{split}
\end{equation}
This immediately implies that
  $A=(a_{ij})_{1\le i,j\le d}$
 is non-negative definite. Furthermore, by \eqref{matix} and \eqref{e1-4}, it is not difficult to verify that
if $\liminf_{|z|\to \infty}\Ee[\kappa(x,z)]>0$ for every $x\in \Z^d$, then $\sum_{i,j=1}^d
  a_{ij}   \xi_i\xi_j>0$
for every non-zero $\xi \in \R^d$, and so $A$ is non-degenerate.
Similarly, if the function $K:\R_{\bf{0}}^d \to \R_+$ in Proposition \ref{t5-4} satisfies that
$K(\theta_0)>0$ for some $\theta_0\in \mathbb{S}^{d-1}$, then $A$ is also non-degenerate. Indeed,
by the continuity of $K$, there exists an open neighborhood of $U\subset \mathbb{S}^{d-1}$ containing $\theta_0$ such that
$\inf_{\theta\in U}K(\theta)>0$. On the other hand, for any non-zero $\xi\in \R^d$, we can find an open subset $U_0\subset U$ satisfying
that $\inf_{\theta\in U_0}|\langle \xi, \theta\rangle|>0$, where $\langle \xi, \theta\rangle:=\sum_{i=1}^d \xi_i\theta_i$.
Therefore, by Proposition \ref{t5-4},
$$\quad
\sum_{i,j=1}^d a_{ij}\xi_i\xi_j =\int_{\mathbb{S}^{d-1}}K(\theta)|\langle \xi, \theta\rangle|^2\,d\theta\ge \int_{U_0}K(\theta)|\langle \xi, \theta\rangle|^2\,d\theta\ge
\inf_{\theta\in U_0}\big(K(\theta)|\langle \xi, \theta\rangle|^2\big)\int_{U_0}d\theta>0.
$$
  \end{itemize}
\end{remark}

\medskip

The rest of the paper is organized as follows. In the next section, we consider a deterministic environment and establish the invariance principle for long range balanced
random walks. The main result for
 $\alpha\in (0,2)$ is Theorem \ref{t5-1}.
Our approach
is based on the tail probability estimates for the exit time of stable-like balanced random walks and
the martingale method.
When $\alpha\in (0,1)$, we can get rid of the balanced condition \eqref{balacd} with some slight modifications of the proof of Theorem \ref{t5-1};
  see Theorem \ref{t5-1*}.
 When $\alpha=2$, the main idea of the proof for Theorem \ref{t5-1} still works; see Theorem \ref{t5-1-1-2}.
 Section \ref{section3} is devoted to the proofs of Theorem \ref{t5-2} and
Proposition
\ref{t5-4}.
Some extensions and remarks of our main results are given in the last section.

\section{Invariance principle for long range balanced random walks}

In this section, we fix the environment $\w$ in $\{\kappa(x,z)(\w): x, z\in \Z^d\}$; in other words, we consider a deterministic environment and discuss the invariance principle for its corresponding purely discontinuous Markov process on
$\Z^d$.

Let $X:=(X_t)_{t\ge 0}$ be a Markov process on $\Z^d$ associated
with the following infinitesimal generator
$$
{\LL} f(x):=\sum_{z\in\Z_{\bf 0}^d}\big(f(x+z)-f(x)\big)\frac{\kappa(x,z)}{|z|^{d+\alpha}},\quad f\in
B_b(\Z^d),
$$
where $\alpha\in (0,2]$
and  $\kappa(\cdot,\cdot): \Z^d\times \Z_{\bf0}^d\to [0,\infty)$.
Denote by $\Pp_x$ the distribution of $X$ on
$\D([0,\infty);\Z^d)$ endowed with the Skorohod topology
and with initial point
 $x\in \Z^d$.

For any
$\alpha\in (0,2]$ and
$n\ge1$, consider the scaled process
$$
X^n:=\{X_t^{n}: t\ge 0\}:=\begin{cases}\{ {n}^{-1}X_{n^\alpha t}:  t\ge 0\},&\quad \alpha\in (0,2),\\
\left\{ {n}^{-1}X_{{n^2}t/{\log (1+n)}}:
t\ge 0\right\},&\quad \alpha=2,\end{cases}
$$
which takes values in $n^{-1}\Z^d$. Clearly $X^n$ is a strong Markov process on $n^{-1}\Z^d$,
and it is easy to check that it has the corresponding infinitesimal
generator
\begin{equation}\label{e:2.4}
{\LL}_n f(x) :=\begin{cases}
n^{-d}
\sum\limits_{z\in n^{-1}\Z_{\bf 0}^d}
\left(f(x+z)-f(x)\right) \frac{\kappa(nx,nz)}{|z|^{d+\alpha}},&\quad \alpha\in (0,2),\\
(n^d\log(1+n))^{-1}
\sum\limits_{z\in n^{-1}\Z_{\bf0}^d }
\left(f(x+z)-f(x)\right) \frac{\kappa(nx,nz)}{|z|^{d+2}},&\quad \alpha=2\end{cases}
\end{equation} acting on $f\in B_b(n^{-1}\Z^d).$
Denote by $\Pp_x^{n}$ the distribution of $X^n$ on
$\D([0,\infty);n^{-1}\Z^d)$
starting at  $x\in n^{-1}\Z^d$.

\subsection{Balanced case for $\alpha\in (0,2)$}

Throughout this subsection, we assume $\alpha\in (0,2)$ and
the balanced condition
\eqref{balacd}.

\medskip

{ \paragraph{{\bf Assumption (B1)}}\it
There exist constants $\theta\in (0,1)$, $C_1>0$ and $R_0\ge 1$
such that for
every $R>R_0$ and $r\in [R^{\theta},  R]$,
\begin{equation}\label{a5-2-1}
\sup_{x\in B({\bf0}, 2R)}\sum_{z\in \Z^d: 1\leq |z|\le r}
\frac{\kappa(x,z)}{|z|^{d+\alpha-2}}\le C_1 r^{2-\alpha},
\end{equation}
and
\begin{equation}\label{a5-2-2}
\begin{split}
\sup_{x \in B({\bf0},2R)}\sum_{z\in \Z^d:
|z|>r}\frac{\kappa(x,z)}{|z|^{d+\alpha}}\le C_1r^{-\alpha}.
\end{split}
\end{equation}
}

\begin{lemma}\label{l5-1}
Suppose that Assumption  {\bf(B1)} holds. Then, we have

{\rm(i)} There is a constant $c>0$ that depends only on
the constant $C_1$ in Assumption  {\bf(B1)}
so that for every  $R>R_0$,
$r\in [R^{\theta},  R]$, $x\in B({\bf0},R)$ and $t>0$,
\begin{equation}\label{l5-1-4}
\Pp_x(\tau_{B(x,r)}\le t)\le c\, t/r^{\alpha}.
\end{equation}
Here and in what follows, for any subset $A\subset \Z^d$,
$\tau_{A}=\inf\{t>0:
X_t\notin A\}$ is the first exit time from $A$ by the process
$X$.

{\rm(ii)}  $\{\Pp_{\bf0}^{n}\}_{n=1}^\infty$ is tight in $\D ([0,T];\R^d)$ for any $T>0$.
\end{lemma}

\begin{proof}
(i) For any $r>0$ and $x\in \R^d$, take $f_{x,r}\in C_b^2(\R^d)$
such that
\begin{align*}
f_{x,r}(z)=
\begin{cases}
0,\ \ & 0\le |z-x|\le r/2,\\
\in [0,1],\ \ &r/2<|z-x|<r,\\
1,\ \ \ &|z-x|\ge r,
\end{cases}
\end{align*}
$\sup_{z\in \R^d}|\nabla f_{x,r}(z)|\le c_0r^{-1}$, and $\sup_{z\in \R^d}|\nabla^2 f_{x,r}(z)|\le c_0r^{-2}$ for some constant
$c_0>0$ independent of $r$ and $x$.
Then, for every $x\in B({\bf0},R)$,
$$
\Pp_x(\tau_{B(x,r)}\le t)\le \Ee_x f_{x,r}(X_{t\wedge
\tau_{B(x,r)}})  =\Ee_x\left[\int_0^{t\wedge \tau_{B(x,r)}}
{\LL}f_{x,r}(X_s)\,ds\right].
$$

Let $R_0\ge1$ be the constant in Assumption {\bf(B1)}. For every
$R>R_0$, $r\in [R^{\theta},  R]$, $x\in B({\bf0},R)$ and $y\in B(x,r)$,
\begin{equation}\label{e:key}\begin{split}
{\LL}f_{x,r}(y)=&\sum_{z\in \Z^d:1\le |z|\le
r}\big(f_{x,r}(y+z)-f_{x,r}(y)-\nabla f_{x,r}(y)\cdot z \big)
\frac{\kappa(y,z)}{|z|^{d+\alpha}}\\
&+\sum_{z\in \Z^d:|z|\ge r}\big(f_{x,r}(y+z)-f_{x,r}(y)\big)
\frac{\kappa(y,z)}{|z|^{d+\alpha}}\\
=&:I_{1,r}+I_{2,r},
\end{split}\end{equation}
where we used \eqref{balacd}
in the first equality.

According to \eqref{a5-2-1}, for every $R>R_0$, $r\in [R^{\theta},  R]$, $x\in B({\bf0},R)$ and $y\in
B(x,r)$,
\begin{align*}
|I_{1,r}|&\le \frac 1 2 \|\nabla^2 f_{x,r}\|_\infty
\sum_{z\in \Z^d: 1\le |z|\le r}\frac{\kappa(y,z)}{|z|^{d+\alpha-2}}\le
\frac {c_0} 2 r^{-2}\sup_{y\in B({\bf0},2R)}\sum_{z\in \Z^d: 1\le |z|\le
r}\frac{\kappa(y,z)}{|z|^{d+\alpha-2}} \le
\frac {c_0 C_1} 2 r^{-\alpha}.
\end{align*}
Similarly, we have by \eqref{a5-2-2},
$$
\sup_{x\in B({\bf0},R), \, y\in B(x,r)}|I_{2,r}|\le 2C_1 r^{-\alpha}.
$$
Hence, for every $R>R_0$ and $r\in [R^{\theta},  R]$,
\begin{equation}\label{e:2.14}
\sup_{x\in B({\bf0},R),y\in B(x,r)}|{\LL}f_{x,r}(y)|\le 2
(1+c_0/4) C_1 r^{-\alpha}.
\end{equation}
Combining all the estimates above,
we obtain \eqref{l5-1-4}.

(ii) For a Borel subset $A\subset
n^{-1}\Z^d$, let $\tau^n_A:=\inf\{t>0:X_t^n \notin A\}$ be the first exit
time from $A$ by the process
$X^n$. By the fact
that $X_t^n=n^{-1}X_{n^\alpha t}$ and \eqref{l5-1-4}, for every
fixed integer $R\geq 1$
and $n>{R_0}/{R}$, we have
\begin{align*}
\Pp^n_{\bf0}\left(\sup_{t\in [0,T]}|X_t^n|>R\right)
&\le \Pp^n_{\bf0} \left(\tau^n_{B({\bf0},R)}\le T\right)
=\Pp_{\bf0}\left(\tau_{B({\bf0},nR)}\le n^\alpha T\right)\le c \, n^\alpha T(nR)^{-\alpha}=c T/R^{\alpha}.
\end{align*}
Consequently,
\begin{equation}\label{l5-1-1}
\lim_{R\rightarrow \infty}\limsup_{n \rightarrow
\infty}\Pp_{\bf0}^n\left(\sup_{t\in [0,T]}|X_t^n|>R\right)=0.
\end{equation}

On the other hand, for any  $\eta>0$, any sequence of stopping times
$\{\tau_n\}_{n\ge1}$ of $\{X^n\}_{n\ge1}$ such that $\tau_n\le T$, and any sequence $\{\e_n\}_{n\ge1}$
such that $\lim_{n \rightarrow \infty}\e_n=0$, it follows from the
strong Markov property of $X^n$ that
\begin{align*}
 \Pp_{\bf0}^n\big(|X_{\tau_n+\e_n}^n-X_{\tau_n}^n|>\eta\big)
&=
\Ee_{\bf0}^n\big[ \Pp^n_{X_{\tau_n}^n}\big(|X^n_{\e_n}-X^n_0|>\eta\big)\big] \\
&\le \Pp_{\bf0}^n\big(\tau_{B({\bf0},R)}^n\le T\big)+\sup_{x\in B({\bf0},R)}\Pp^n_x\big(\tau_{B(x,\eta)}^n\le \e_n\big)\\
&=\Pp_{\bf0}\big(\tau_{B({\bf0},nR)}\le n^\alpha T\big)+ \sup_{x\in
B({\bf0},nR)}\Pp_x\big(\tau_{B({\bf0},n\eta)}\le n^\alpha\e_n\big),
\end{align*}
where in the first inequality we used the fact $\tau_n\le T$ for the second term. Taking
$n$ large enough so that $nR>R_0$ and $n\eta>(n R)^\theta$, we get
from \eqref{l5-1-4} that
$$\limsup_{n \rightarrow \infty}\Pp_{\bf0}^n\big(|X_{\tau_n+\e_n}^n-X_{\tau_n}^n|>\eta\big)\\
\le c\limsup_{n \rightarrow \infty} \left(\frac{n^\alpha
T}{(nR)^\alpha} +\frac{n^\alpha \e_n}{(n\eta)^\alpha}\right)\le
cT/R^{\alpha},
$$
which, by taking $R\to \infty$, yields
\begin{equation}\label{l5-1-2}
\limsup_{n\rightarrow
\infty}\Pp_{\bf0}^n\big(|X^n_{\tau_n+\e_n}-X^n_{\tau_n}|>\eta\big)=0.
\end{equation}
The desired tightness assertion now follows from \eqref{l5-1-1}, \eqref{l5-1-2} and \cite[Theorem 1]{Ad}.
\end{proof}

\begin{lemma}\label{l5-2}
Under Assumption {\bf(B1)}, for every $f\in C_c^2(\R^d)$,
\begin{equation}\label{l5-2-1}
\sup_{n\geq 1} \sup_{x\in n^{-1}\Z^d}|{\LL}_nf(x)|<\infty,
\end{equation} and
\begin{equation}\label{l5-2-1a}
\lim_{R \rightarrow \infty}\limsup_{n \rightarrow \infty}\sup_{x\in
n^{-1}\Z^d: |x|\ge R}|{\LL}_n f(x)|=0.
\end{equation}
 \end{lemma}

 \begin{proof}
 Fix $f\in C_c^2(\R^d)$.
Suppose that $\text{supp}(f)\subset B({\bf0},N_0)$ for some $N_0\ge1$.
Then for every $x\in B({\bf0},4N_0)\cap n^{-1}\Z^d$, we have by \eqref{e:2.4} that
for $n$ large enough,
\begin{equation}\label{eq:grad2e}
\begin{split}
|{\LL}_n f(x)|
&\le   n^{-d}\sum_{z\in n^{-1}\Z_{\bf0}^d: |z|\le 1}\Big|f(x+z)-f(x)-\nabla f(x)\cdot z\Big|\frac{\kappa(nx,nz)}{|z|^{d+\alpha}} \\
&  \quad +n^{-d}\sum_{z\in n^{-1}\Z^d: |z|> 1}\Big|f(x+z)-f(x)\Big|\frac{\kappa(nx,nz)}{|z|^{d+\alpha}} \\
&\le
\frac 1 2 n^{\alpha-2}\|\nabla^2 f\|_\infty\sup_{x\in B({\bf0},4nN_0)}
\sum_{z\in \Z_{\bf0}^d: |z|\le n}\frac{\kappa(x,z)}{|z|^{d+\alpha-2}} \\
&  \quad + 2n^{\alpha}\|f\|_\infty \sup_{x\in B({\bf0},4nN_0)}\sum_{z\in
\Z^d: |z|>n}\frac{\kappa(x,z)}{|z|^{d+\alpha}} \\
&\leq
\frac {C_1} 2 n^{\alpha-2}\|\nabla^2 f\|_\infty \, n^{2-\alpha} + 2C_1 n^{\alpha}\|f\|_\infty n^{-\alpha}
  =   2C_1 (\|\nabla^2 f\|_\infty + \| f \|_\infty),
\end{split}
\end{equation}
where we used \eqref{balacd} in the first inequality, and
\eqref{a5-2-1} and \eqref{a5-2-2} in the last inequality.

On the other hand, since $f$ is supported in $B({\bf0}, N_0)$, if $|x|>4N_0$, then $f(x+z)-f(x)= 0$ when
$|x+z|>N_0$. Hence for any $|x|>4N_0$ and $n$ large enough, we have by \eqref{e:2.4} that
\begin{equation}\label{l5-2-2}
\begin{split}
|{\LL}_n f(x)|
&\le \|f\|_\infty \, n^{-d} \sum_{z\in n^{-1}\Z_{\bf0}^d: \, |x+z|\leq N_0}
\frac{\kappa(nx,nz)}{|z|^{d+\alpha}}\le \|f\|_\infty \, n^{-d}\sum_{z\in n^{-1}\Z^d: \,
|x|/2\le |z|\le 2|x|}\frac{\kappa(nx,nz)}{|z|^{d+\alpha}}\\
&\le \|f\|_\infty \, n^{\alpha} \sum_{z\in \Z^d: \, {n|x|}/{2}\le |z|\le 2n|x|}\frac{\kappa(nx,z)}{|z|^{d+\alpha}}\le 4\|f\|_\infty \, n^\alpha |nx|^{-2}  \sum_{z\in \Z_{\bf 0}^d: \, |z|\le 2n|x|}
\frac{\kappa(
nx,z)}{|z|^{d+\alpha-2}}\\
&\le 4C_1 \|f\|_\infty \, n^{\alpha-2}  |x|^{-2} (2n|x|)^{2-\alpha} =   c_1\|f\|_\infty|x|^{-\alpha},
\end{split}
\end{equation}
where we have used  \eqref{a5-2-1} in the last inequality. This proves
\eqref{l5-2-1a}, and along with \eqref{eq:grad2e} also  yields
\eqref{l5-2-1}.
 \end{proof}

 We need the following assumption for the convergence of
 $\{X^n\}_{n\ge1}$.

\medskip

{ \paragraph{{\bf Assumption (B2)}}  \it
There exists a bounded continuous function $K(x,z):\R^d\times
\R_{\bf0}^d\rightarrow (0,\infty)$ such that $K(x,z)=K(x,-z)$ for all $x\in \R^d$ and $z\in
\R_{\bf0}^d$, and that
 for  any integer $R\geq 1$
 and any $f\in C_c^2(\R^d)$,
\begin{equation}\label{a5-3-1}
  \liminf_{\varepsilon \to 0}
  \lim_{n \rightarrow \infty}
n^\alpha\sup_{x\in \Z^d: \atop{ |x|\le nR}}
\Bigg| \sum_{z\in \Z^d:\atop {
 n \e <|z|<n/\e } }
\left(f\Big(\frac{x+z}{n}\Big)-f\left(\frac{x}{n}\right)\right)
\left(\frac{\kappa(x,z)-K(\frac{x}{n},\frac{z}{n})}{|z|^{d+\alpha}}\right) \Bigg| =0.
\end{equation}
}

\medskip

Clearly  in the above assumption, the phase
``for
any  integer $R\geq 1$'' can be replaced by ``for
any constant $R>0$''.

\medskip

\begin{lemma}\label{l5-3}
Suppose that Assumptions {\bf (B1)} and {\bf(B2)} hold. Then for
any $f\in C_c^2(\R^d)$,
\begin{equation}\label{l5-3-1}
\lim_{n \rightarrow \infty}\sup_{x\in n^{-1}\Z^d}|{\LL}_nf(x)-\bar
{\LL}f(x)|=0,
\end{equation}
where the operator $\bar {\LL}$ is defined by \eqref{l5-3-1a}.
\end{lemma}

\begin{proof}
For every $\eta\in (0,1)$, set for $x\in n^{-1}\Z^d$
\begin{align*}
{\LL}_{n,\eta}f(x):&=n^{-d}\sum_{z\in n^{-1}\Z^d: \atop{\eta<|z|<1/\eta }}
\big(f(x+z)-f(x)\big)\frac{\kappa(nx,nz)}{|z|^{d+\alpha}} \\
&=
n^\alpha\sum_{z\in \Z^d: n\eta<|z|<n\eta^{-1}}
\Big(f\Big(x+\frac{z}{n}\Big)-f(x)\Big)\frac{\kappa(nx,z)}{|z|^{d+\alpha}},\\
\bar {\LL}_{n,\eta}f(x):&=n^{-d}\sum_{z\in n^{-1}\Z^d:
\atop{\eta<|z|<1/\eta}}
\big(f(x+z)-f(x)\big)\frac{K(x,z)}{|z|^{d+\alpha}} \\
&=
n^\alpha\sum_{z\in \Z^d:
n\eta<|z|<n\eta^{-1}}
\Big(f\Big(x+\frac{z}{n}\Big)-f(x)\Big)\frac{K(x, {z}/{n})}{|z|^{d+\alpha}}
\end{align*}
and for $x\in \R^d$,
\begin{equation}\label{l5-3-2}
\bar {\LL}_\eta
f(x) :=\int_{\{z\in \R^d:\eta<|z|<\eta^{-1}\}}\big(f(x+z)-f(x)\big)\frac{K(x,z)}{|z|^{d+\alpha}}\,dz.
\end{equation}
For  every $R>1$ and $\eta\in (0,1)$,
\begin{align}
&\sup_{x\in n^{-1}\Z^d}|{\LL}_nf(x)-\bar {\LL}f(x)|\nonumber\\
&\le \sup_{x\in n^{-1}\Z^d: {|x|\le R}}|{\LL}_{n,\eta}f(x)-\bar {\LL}_\eta
f(x)|+
\sup_{x\in n^{-1}\Z^d: {|x|>R}}|{\LL}_n f(x)|+ \sup_{x\in n^{-1}\Z^d:
{|x|>R}}|\bar {\LL} f(x)|\nonumber\\
&\quad+n^{-d}
\sup_{x\in n^{-1}\Z^d: {|x|\le R}}\Big|\sum_{z\in
n^{-1}\Z_{\bf 0}^d : {|z|\le \eta}}
\big(f(x+z)-f(x)-\nabla f(x)\cdot z\big)\frac{\kappa(nx,nz)}{|z|^{d+\alpha}}\Big|\nonumber\\
&\quad+n^{-d}
\sup_{x\in n^{-1}\Z^d: {|x|\le R}} \Big|\sum_{z\in
n^{-1}\Z^d: {|z|\ge \eta^{-1}}}
\big(f(x+z)-f(x)\big)\frac{\kappa(nx,nz)}{|z|^{d+\alpha}}\Big|\nonumber\\
&\quad+
\sup_{x\in n^{-1}\Z^d: {|x|\le R}}\Big|\int_{\{0<|z|\le
\eta\}\cup \{ |z|\ge \eta^{-1}\}}
\big(f(x+z)-f(x)-\nabla f(x)\cdot z
\I_{\{|z|\le 1\}}
\big)\frac{K(x,z)}{|z|^{d+\alpha}}\,dz\Big|\nonumber\\
&=:I_1^{n,R,\eta}+I_2^{n,R}+I_3^{n,R}+I_4^{n,R,\eta}+I_5^{n,R,\eta}+I_6^{n,R,\eta}.\label{eq:i1-i6}
\end{align}
Note that due to
balanced
conditions
\eqref{balacd}
and \eqref{K-ban},
we may add the gradient term
$\nabla f(x)\cdot z$ in the summation.

By \eqref{a5-3-1}, there is a  sequence of positive numbers $ \{\e_k\}_{k\ge1} \subset (0,1)$ that decreases to $0$
so that
\begin{equation}\label{e:2.17}
    \lim_{k\to \infty}
  \lim_{n \rightarrow \infty}
n^\alpha\sup_{x\in \Z^d: \atop{ |x|\le nR}}
\Bigg| \sum_{z\in \Z^d:\atop {
   n \e_k<|z|<n/\e_k} } \left(f\Big(\frac{x+z}{n}\Big)-f\left(\frac{x}{n}\right)\right)
\left(\frac{\kappa(x,z)-K(\frac{x}{n},\frac{z}{n})}{|z|^{d+\alpha}}\right) \Bigg| =0.
\end{equation}
Thus we have
$$
\liminf_{k\to \infty}\lim_{n \rightarrow \infty}\sup_{x\in n^{-1}\Z^d: |x|\le
R}|{\LL}_{n,  \e_k}f(x)-\bar {\LL}_{n, \e_k}f(x)|=0.
$$
Since $K(x,z)$ is uniformly continuous on $\{(x,z)\in \R^{2d}:
|x|\le R \textrm{ and }  \e_k<|z|< \e_k^{-1}\}$  for fixed $k\ge1$,  it is routine to show
that for  any $f\in C_c^2(\R^d)$  and $k\ge1$,
$$
\lim_{n \rightarrow \infty}\sup_{x\in n^{-1}\Z^d: |x|\le R}|\bar
{\LL}_{n, \e_k}f(x)-\bar {\LL}_{ \e_k}f(x)|=0.
$$
Hence, for any $R>1$,  $$
\liminf_{k\to \infty}\lim_{n\rightarrow \infty}I_1^{n,R, \e_k}=0.
$$

On the other hand, by \eqref{l5-2-1a},
$$
\lim_{R \rightarrow \infty}\limsup_{n \rightarrow \infty}I_2^{n,R}=0.
$$
Following the proof for \eqref{l5-2-1},
and applying \eqref{a5-2-1}
and \eqref{a5-2-2} respectively, we can get
$$\limsup_{n \rightarrow \infty}
I_4^{n,R,\eta}
\le c_1\eta^{2-\alpha},~~~\mbox{ and }~~~
\limsup_{n \rightarrow\infty}
I_5^{n,R,\eta}\le c_1\eta^{\alpha}.
$$
Since $K$ is bounded, it is obvious that
$$\lim_{R\rightarrow \infty}\limsup_{n \rightarrow
\infty}I_3^{n,R}=0,
~~~\mbox{ and }~~~
\limsup_{n \rightarrow\infty}
I_6^{n,R,\eta}\le
c_2\big(\eta^{2-\alpha}+\eta^\alpha\big).
$$
 Now, we take $\eta=\e_k$ in the estimate \eqref{eq:i1-i6}. Combining all estimates above
with \eqref{eq:i1-i6},
first letting $n \rightarrow \infty$, then taking $R \rightarrow \infty$ and
$k\to\infty$,
we obtain
\eqref{l5-3-1}.
\end{proof}

Now, we can state the main result in this subsection.

\begin{theorem}\label{t5-1}
Suppose that Assumptions {\bf(B1)} and {\bf(B2)} hold, and the solution of the  martingale problem for $(\bar {\LL}, C^2_c(\R^d))$
defined by \eqref{l5-3-1a} is unique. Then, for every $T>0$, $\Pp^{n}_{\bf0}$ converges weakly to $\bar \Pp_{\bf0}$, where $\bar \Pp_{\bf0}$ denotes the
distribution $($restricted on the time interval $[0,T]$$)$ of the
unique solution to the martingale problem of the operator $(\bar {\LL}, C^2_c(\R^d))$.
\end{theorem}

\begin{proof}
According to Lemma \ref{l5-1}, $\{\Pp^{n}_{\bf0}\}_{n\ge 1}$ is tight.
Then, there exists a weakly convergent subsequence
$\{\Pp^{n_k}_{\bf0}\}_{k\ge 1}$ (which we also denote by
$\{\Pp^{n}_{\bf0}\}_{n\ge 1}$ for simplicity) with a weak limit $\tilde
\Pp_{\bf0}$. By our assumption, the solution of martingale problem for $(\bar {\LL}, C_c^2 (\R^d))$ is
unique. Therefore, it suffices to prove the weak limit $\tilde \Pp_{\bf0}$
solves the martingale problem for $(\bar \LL, C^2_c (\R^d))$
with initial value ${\bf0}$.

 By the Skorohod theorem,
we can find a probability space $(\Omega, \mathscr{F},\hat \Pp)$ and a family of processes
 $(X_t^n)_{t\ge 0}$ and $(X_t)_{t\ge 0}$ on this space such that
\begin{itemize}
\item [(i)] The laws of $(X_t^n)_{t\ge 0}$ and $(X_t)_{t\ge 0}$ under
$\hat \Pp$ are $\Pp^{n}_{\bf0}$ and $\tilde \Pp_{\bf0}$, respectively;
\item [(ii)] For every $T>0$,
$\hat \Pp$-a.s.\
the process
  $X_{\cdot}^n$ converges to  $X_{\cdot}$ on $\D([0, T];\R^d)$.
\end{itemize}

Since $(X_t^n)_{t\ge0}$ is a solution to the martingale problem of
  $({\LL}_n, B_b(n^{-1}\Z^d))$,
 for every $0\le s_1\le \cdots\le s_k\le s\le t\le
T$, $f\in C_c^2(\R^d)$ and $G\in C_b((\R^d)^k)$ with $k\ge1$,
\begin{equation}\label{t5-1-1}
\hat \Ee\left[\left(f(X_t^n)-f(X_s^n)-\int_s^t {\LL}_n f(X_r^n)\,dr\right)G\left(X_{s_1}^n,\cdots X_{s_k}^n\right)\right]=0.
\end{equation}
Note that by \eqref{l5-2-1}, the random variable inside the above expectation is uniformly bounded in $n\geq 1$.
As
 $\hat \Pp$-a.s.\ the process $X_{\cdot}^n$ converges to  $X_{\cdot}$ on
$\D([0,T];\R^d)$,
 by the bounded convergence
theorem,
$$
\lim_{n \rightarrow \infty}\hat \Ee\Big[\Big|G\big(X_{s_1}^n,\cdots X_{s_k}^n\big)-
G\big(X_{s_1},\cdots X_{s_k}\big)\Big|\Big]=0.
$$
On the other hand,
\begin{align*}
 &\hat \Ee\left[\left|\int_s^t {\LL}_nf(X_r^n)\,dr-\int_s^t \bar {\LL} f(X_r)\,dr\right|\right] \\
&\le (t-s)\sup_{x\in n^{-1}\Z^d}|{\LL}_nf(x)-\bar {\LL}f(x)|
+\hat \Ee \left[\left|\int_s^t \bar {\LL}f(X_r^n)dr-\int_s^t \bar {\LL}
f(X_r)\,dr\right|\right].
\end{align*}

We next claim that $\bar {\LL}f\in C_b(\R^d)$ for any $f\in C_c^2(\R^d)$.
Indeed, it  follows from the boundedness of $K(x,z)$ that $\bar {\LL}f$ is bounded in $\R^d$ for any $f\in C_c^2(\R^d)$.
On the other hand, by the continuity and the boundedness of
$K(x,z)$, we know $\bar {\LL}_\eta f\in C_b(\R^d)$ for all
$f\in C_c^2(\R^d)$ and $\eta>0$, where $\bar {\LL}_\eta f$ is defined by \eqref{l5-3-2}.
Following the proof of Lemma \ref{l5-3}, we can obtain
$$
\sup_{x\in \R^d}|\bar {\LL} f(x)-\bar {\LL}_\eta f(x)|\le c_1\left(\eta^{2-\alpha}+\eta^{\alpha}\right).
$$
Hence $\bar \LL f$ is the uniform limit of $\bar \LL_\eta$ as $\eta \to 0$, and it follows that
$\bar \LL f \in C_b (\R^d)$ for any $f\in C_c^2(\R^d)$.

By the bounded convergence theorem,
$$
\lim_{n \rightarrow \infty}\hat \Ee
\left[\left|\int_s^t \bar {\LL}f(X_r^n)\,dr-\int_s^t \bar {\LL} f(X_r)\,dr\right|\right]=0,
$$
which together with
\eqref{l5-3-1} yields
$$
\lim_{n \rightarrow \infty}\hat \Ee\left[\left|\int_s^t
{\LL}_nf(X_r^n)\,dr-\int_s^t \bar {\LL} f(X_r)\,dr\right|\right]=0.
$$
Putting all the estimates above into \eqref{t5-1-1}, and letting
$n\rightarrow \infty$,  we
obtain
$$
\hat \Ee\left[\left(f(X_t)-f(X_s)-\int_s^t \bar {\LL} f(X_r)\,dr\right)G\left(X_{s_1},\cdots X_{s_k}\right)\right]=0.
$$
This shows that
$\tilde \Pp_{\bf0}$
is a solution to the martingale problem $(\bar \LL, C^2_c(\R^d))$
and clearly
$\tilde \Pp_{\bf0}(X_0={\bf0})=1$.
 This completes the  proof.
\end{proof}

\begin{remark}\label{keyrem}
As the above proofs show, we can replace the balanced condition
\eqref{balacd}
by
the following slightly weaker condition: there is some large $r_0\geq 1$ so that for all
$r\geq r_0$
and $x\in \Z^d$,
 \begin{equation}\label{e:zero}
\sum_{z\in \Z_{\bf0}^d :\, |z|\le r} z\kappa(x,z)/
|z|^{d+\alpha}=0.\end{equation}
 Note that under \eqref{e:zero} the generator of the process $X$ can be written as
$$\LL f(x)=\sum_{z\in \Z_{\bf0}^d: |z|\le r} (f(x+z)-f(x)-\nabla f(x)\cdot z)\frac{\kappa(x,z)}{|z|^{d+\alpha}}+ \sum_{z\in \Z_{\bf0}^d: |z|> r} (f(x+z)-f(x))\frac{\kappa(x,z)}{|z|^{d+\alpha}}
$$
 for all $f\in B_b(\Z^d)$ and $r\ge r_0$. This indicates that the process $X$ is almost
 driftless in large scale, and so the drift term does not contribute
  to the scaling process. On the other hand, when $\alpha\in (0,1)$ the tail of long range jumps for the process $X$ is up to the order $r^{-\alpha}$, which would dominate
 the drift term when we do the scaling. Based on this observation, we can expect that, when $\alpha\in (0,1)$, the invariance principle still holds without the balanced condition \eqref{balacd} (or \eqref{e:zero}). The details are
 given in the next subsection.
\end{remark}

\subsection{Non-balanced case for $\alpha\in (0,1)$}\label{subsecnonba}

In this subsection, we restrict ourselves to the case
$0<\alpha < 1$, and obtain similar results
as in the previous subsection
but without the balanced condition \eqref{balacd}.

\medskip

{ \paragraph{{\bf Assumption (B1${^*}$)}}\it
There exist constants $\theta\in (0,1)$, $C_1>0$ and $R_0\ge1$
such that  for
 every $R>R_0$ and $r\in [{R^{\theta} }, \,  R]$,
\begin{equation}\label{a5-2-1-1}
\sup_{x\in B({\bf0}, 2R)}\sum_{z\in \Z_{\bf0}^d : |z|\le r}
\frac{\kappa(x,z)}{|z|^{d+\alpha-1}}\le C_1 r^{1-\alpha},
\end{equation}
and
\begin{equation}\label{a5-2-2*}
\begin{split}
\sup_{x \in B({\bf0},2R)}\sum_{z\in \Z^d:
|z|>r}\frac{\kappa(x,z)}{|z|^{d+\alpha}}\le C_1r^{-\alpha}.
\end{split}
\end{equation}
}

Clearly, condition \eqref{a5-2-1-1} implies condition \eqref{a5-2-1},
while \eqref{a5-2-2*}
is the same as
\eqref{a5-2-2} in Assumption {\bf(B1)}.

\medskip

The following result corresponds  to Theorem \ref{t5-1} in the balanced conductance case
for $\alpha\in (0,2)$.

\begin{theorem}\label{t5-1*} Let $0<\alpha < 1$. Suppose that Assumptions {\bf(B1${^*}$)} and {\bf(B2)} hold, and the solution of the  martingale problem for $(\bar {\LL}, C^2_c(\R^d))$ defined by \eqref{l5-3-1a} is unique.
Then the conclusion of Theorem $\ref{t5-1}$ holds.
\end{theorem}

\begin{proof}
In the proof of Theorem \ref{t5-1} and related lemmas,
there are three
places
where condition \eqref{balacd} is used.
Below, we explain how to
modify
the corresponding parts of the proof.
The rest part
goes through exactly as that for  Theorem \ref{t5-1}.

The first part is the estimates of \eqref{e:key} in the proof of Lemma \ref{l5-1}.
In the current non-balanced case,
\begin{align*}
{\LL}f_{x,r}(y)=&\sum_{z\in \Z_{\bf0}^d: \,  |z|\le
r}\big(f_{x,r}(y+z)-f_{x,r}(y)\big)
\frac{\kappa(y,z)}{|z|^{d+\alpha}}+\sum_{z\in \Z^d:|z|\ge r}\big(f_{x,r}(y+z)-f_{x,r}(y)\big)
\frac{\kappa(y,z)}{|z|^{d+\alpha}}\\
=&:I_{1,r}+I_{2,r}.
\end{align*}
By \eqref{a5-2-1-1}, for every for
every $R>R_0$, $r\in [{R^{\theta} }, \,  R]$,
$x\in B({\bf0},R)$ and $y\in B(x,r)$,
\begin{align*}
 |I_{1,r}|&\le \|\nabla f_{x,r}\|_\infty
 \sum_{z\in \Z^d_{\bf0}: |z|\le r}\frac{\kappa(y,z)}{|z|^{d+\alpha-1}}\le c_2r^{-1}\sup_{y\in B({\bf0},2R)}\sum_{z\in \Z^d_{\bf0}: \, |z|\le
r}\frac{\kappa(y,z)}{|z|^{d+\alpha-1}} \le c_3r^{-\alpha}.
\end{align*}
By the same way as that in the proof of Lemma \ref{l5-1}, we have
$|I_{2, r}|\leq c_4r^{-\alpha}$ for all $x\in B({\bf0}, R)$ and $y\in B(x, r)$.
Thus \eqref{e:2.14} holds and, consequently,
\eqref{l5-1-4} holds.

The second part is
\eqref{eq:grad2e} and \eqref{l5-2-2}
in the proof of Lemma \ref{l5-2}.
In the non-balanced case, estimate \eqref{eq:grad2e} can be done in the following way. For $n$ large enough,
in view of \eqref{e:2.4}, \eqref{a5-2-1-1} and \eqref{a5-2-2*},
\begin{equation}\label{e:2.27}
\begin{split}
  |{\LL}_n f(x)|
&\leq   n^{-d} \sum_{z\in n^{-1}\Z_{\bf0}^d : |z|\leq 1}
\| \nabla f\|_\infty \, |z| \, \frac{\kappa(nx, nz)}{|z|^{d+\alpha}}
 + n^{-d}\sum_{z\in n^{-1}\Z^d: |z|\geq 1}
2 \| f\|_\infty \, \frac{\kappa(nx, nz)}{|z|^{d+\alpha}} \\
&\le  n^{\alpha-1}\|\nabla f\|_\infty
\sum_{z\in \Z_{\bf 0}^d : |z|\le n}\frac{\kappa(nx,z)}{|z|^{d+\alpha-1}}
+ 2n^{\alpha} \|f\|_\infty  \sum_{z\in \Z^d: |z|>n}\frac{\kappa(nx,z)}{|z|^{d+\alpha}}\\
&\leq   C_1 \| \nabla f\|_\infty + 2 C_1 \|f\|_\infty.
\end{split}
\end{equation} Hence \eqref{eq:grad2e} holds.
On the other hand,
the fourth inequality of \eqref{l5-2-2}
should be
 replaced by
$$4\|f\|_\infty \, n^\alpha |nx|^{-1}  \sum_{z\in \Z_{\bf0}^d: \, |z|\le 2n|x|}
\frac{\kappa(
nx,z)}{|z|^{d+\alpha-1}}.$$
Then, by \eqref{a5-2-1-1},  the end estimate in \eqref{l5-2-2} holds.

The third place is the estimates for $I_4^{n,R,\eta}$ and $I_6^{n,R,\eta}$ in \eqref{eq:i1-i6}.
Without the balanced condition, we can not insert the item $\nabla f(x)\cdot z$ into the definitions of
$I_4^{n,R,\eta}$ and $I_6^{n,R,\eta}$.
Since $0<\alpha<1$, we do it directly. That is,
$$
I_4^{n,R,\eta}=n^{-d}
\sup_{x\in n^{-1}\Z^d: {|x|\le R}}\Big|\sum_{z\in
n^{-1}\Z^d_{\bf0}: {|z|\le \eta}}
\big(f(x+z)-f(x)\big)\frac{\kappa(nx,nz)}{|z|^{d+\alpha}}\Big|.
$$
Then, by the same argument as \eqref{e:2.27}, we get
$$
\lim_{\eta \downarrow 0}\limsup_{n \rightarrow \infty}I_4^{n,R,\eta}\le c_5\lim_{\eta \downarrow 0}\eta^{1-\alpha}=0.
$$
Similarly, we have  $$
\lim_{\eta \downarrow 0}\limsup_{n \rightarrow \infty}I_6^{n,R,\eta}\le c_6
\lim_{\eta \downarrow 0}(\eta^{1-\alpha}+\eta^{\alpha})=0.
$$
With the above modifications, the proof of Lemma \ref{l5-3} goes through. Hence the conclusion of
Theorem  \ref{t5-1} holds under the condition of this theorem.
\end{proof}

\subsection{Balanced case
for $\alpha=2$}\label{subsecba2}
In this subsection,
we consider the case
$\alpha=2$ under the balanced condition \eqref{balacd}. We will make the following two assumptions instead of Assumptions {\bf(B1)} and {\bf(B2)}, respectively.

\medskip

{\paragraph{{\bf Assumption (C1)}}\it
There exist constants $\theta\in (0,1)$, $C_1>0$ and $R_0\ge1$ such that
for every $R>R_0$ and $r\in [R^{\theta},  R]$,
\begin{equation}
\label{e:2.28}
\sup_{x\in B({\bf0}, 2R)}\sum_{z\in \Z^d: 1\leq |z|\le r}
\frac{\kappa(x,z)}{|z|^{d}}\le C_1\log (1+r),
\end{equation}
and
\begin{equation}\label{a5-2-2-2}
\begin{split}
\sup_{x \in B({\bf0},2R)}\sum_{z\in \Z^d:
|z|>r}\frac{\kappa(x,z)}{|z|^{d+2}}\le C_1r^{-2}.
\end{split}
\end{equation}
}

\medskip

{ \paragraph{{\bf Assumption (C2)}}  \it
 For any $n\ge1$, there exists a
function
$\Phi_n(x,z):
n^{-1}\Z^d \times
n^{-1}\Z_{\bf0}^d\rightarrow (0,\infty)
$
with $\Phi_n(x,z)=\Phi_n(x,-z)$ for all $x \in
n^{-1}\Z^d$ and $z\in
n^{-1}\Z_{\bf0}^d$,  so that $\sup_{n\ge 1}\|\Phi_n\|_\infty<\infty$, and
for any
integer $R\geq 1$
and any $f\in C_c^2(\R^d)$,
\begin{equation}\label{a5-3-1-2}
\lim_{n \rightarrow \infty}
\frac{n^2}{\log (1+n)}\sup_{x\in \Z^d: \atop{ |x|\le nR}}
\Bigg| \sum_{z\in \Z^d:\atop {1\le |z|\le n} }
\!\!\left(f\Big(\frac{x+z}{n}\Big)-f\left(\frac{x}{n}\right)\!\right)
\!\!\left(\frac{\kappa(x,z)-
\Phi_n({x}/{n},{z}/{n})}{|z|^{d+2}}\!\right) \Bigg| =0,
\end{equation}
and that for any
integer $R\ge 1$ and $1\le i,j\le d$,
\begin{equation}\label{a5-3-1-1}
\lim_{n\to\infty}\sup_{x\in n^{-1}\Z^d: |x|\le R}\Bigg|
\frac{1}{\log (1+n)}\sum_{z\in \Z_{\bf0}^d: |z|\le n} \frac{z_iz_j
\Phi_n(x,{z}/{n})}{|z|^{d+2}}-a_{ij}\Bigg|=0
\end{equation}
for some constant matrix $A:=(a_{ij})_{1\le i,j\le d}$ on $\R^d$.
}

\medskip

Corresponding to Theorem \ref{t5-1}, we have the following
result.

 \begin{theorem}\label{t5-1-1-2} Let $\alpha=2$. Assume that the balanced condition \eqref{balacd}, and Assumptions {\bf(C1)} and {\bf(C2)} hold.
Then the conclusion of Theorem $\ref{t5-1}$ holds with
\begin{equation}\label{l5-3-1a-2}\bar \LL f(x):=\frac{1}{2}\sum_{1\le i,j\le d} a_{ij}
\frac{\partial^2 f(x)}{\partial x_i \partial x_j},
\quad f\in C_c^2(\R^d).\end{equation}
\end{theorem}
\begin{proof}

 The proof follows from that of Theorem \ref{t5-1} with some necessary modifications. For the convenience of
the reader, we
highlight all main differences here.

(1) First, under Assumption  {\bf(C1)},
there is a constant $c>0$ that depends only on
the constant $C_1$ in Assumption  {\bf(C1)}
so that for every  $R>R_0$,
$r\in [R^{\theta},  R]$, $x\in B({\bf0},R)$ and $t>0$,
$$
\Pp_x(\tau_{B(x,r)}\le t)\le c\, tr^{-2}\log(1+r).
$$
Consequently, $\{\Pp_{\bf0}^{n}\}_{n=1}^\infty$ is tight in $\D ([0,T];\R^d)$ for any $T>0$.

Under Assumption {\bf (C1)}, we
have that for every $f\in C_c^2(\R^d)$,
\begin{equation}\label{l5-2-1-2}
\sup_{n\geq 1} \sup_{x\in n^{-1}\Z^d}|{\LL}_nf(x)|<\infty,
\quad
\lim_{R \rightarrow \infty}\limsup_{n \rightarrow \infty}\sup_{x\in
n^{-1}\Z^d: |x|\ge R}|{\LL}_n f(x)|=0.
\end{equation}
The proofs of the above conclusions are similar to these of Lemmas \ref{l5-1} and \ref{l5-2},
so they are omitted.

\smallskip

(2) Next, we claim that under Assumptions {\bf (C1)} and {\bf(C2)}, for
any $f\in C_c^3(\R^d)$,
\begin{equation}\label{l5-3-1-2}
\lim_{n \rightarrow \infty}\sup_{x\in n^{-1}\Z^d}|{\LL}_nf(x)-\bar
{\LL}f(x)|=0,
\end{equation}
where the operator $\bar {\LL}$ is defined by \eqref{l5-3-1a-2}. Indeed, for any $n\ge1$ and $x\in n^{-1}\Z^d$, define
\begin{align*}
{\LL}_{n,1}f(x):&=\frac{1}{n^d\log(1+n)}\sum_{z\in n^{-1}\Z^d: { 0<|z|\le 1 }}
\big(f(x+z)-f(x)\big)\frac{\kappa(nx,nz)}{|z|^{d+2}} \\
&=
\frac{n^2}{\log(1+n)}\sum_{z\in \Z^d: 1\le |z|\le n}
\Big(f\Big(x+\frac{z}{n}\Big)-f(x)\Big)\frac{\kappa(nx,z)}{|z|^{d+2}},\\
\bar {\LL}_{n,1}f(x):&=\frac{1}{n^d\log(1+n)}\sum_{z\in n^{-1}\Z^d:
 { 0<|z|\le 1}}
\big(f(x+z)-f(x)\big)\frac{ \Phi_n(x,z)}{|z|^{d+2}} \\
&=
\frac{n^2}{\log(1+n)}\sum_{z\in \Z^d:
1\le |z|\le n}
\Big(f\Big(x+\frac{z}{n}\Big)-f(x)\Big)\frac{ \Phi_n(x, {z}/{n})}{|z|^{d+2}}.
\end{align*}

For  every $R>1$,
\begin{align*}
\sup_{x\in n^{-1}\Z^d}|{\LL}_nf(x)-\bar {\LL}f(x)|&\le \sup_{x\in n^{-1}\Z^d: {|x|\le R}}|{\LL}_{n,1}f(x)-\bar {\LL}
f(x)|+
\sup_{x\in n^{-1}\Z^d: {|x|\le R}}|{\LL}_nf(x)-{\LL}_{n,1}f(x)|\\
&\quad+
\sup_{x\in n^{-1}\Z^d: {|x|>R}}|{\LL}_n f(x)|+ \sup_{x\in \R^d:
{|x|>R}}|\bar {\LL} f(x)|\\
&=:I_1^{n,R}+I_2^{n,R}+I_3^{n,R}+I_4^{R}.
\end{align*}

According to
\eqref{a5-3-1-2}, for every $R>1$,
$$
\lim_{n \rightarrow \infty}\sup_{x\in n^{-1}\Z^d: |x|\le
R}|{\LL}_{n,1}f(x)-\bar {\LL}_{n,1}f(x)|=0.
$$
On the other hand, for any $f\in C_c^3(\R^d)$ and $x\in n^{-1}\Z^d$ with $|x|\le R$
\begin{align*}|\bar
{\LL}_{n,1}f(x)-\bar {\LL}f(x)|
&=\left|\frac{n^2}{\log(1+n)}\sum_{z\in \Z^d:
1\le|z|\le n}
\Big(f\Big(x+\frac{z}{n}\Big)-f(x)-\nabla f(x)\cdot \frac{z}{n} \Big)\frac{ \Phi_n(x,{z}/{n})}{|z|^{d+2}}-\bar {\LL}f(x)\right|\\
&\le \left|\frac{1}{2\log(1+n)}\sum_{z\in \Z^d:
1\le |z|\le n} \langle
 \nabla^2 f(x), z\otimes z\rangle \frac{  \Phi_n(x,{z}/{n})}{|z|^{d+2}}-\bar {\LL}f(x)\right|\\
 &\quad + \frac{c_1\|\nabla^3 f\|_\infty}{n\log(1+n)}\sum_{z\in \Z^d:
1\le |z|\le n}\frac{|z|^3}{|z|^{d+2}}.
\end{align*} Here in the equality above we
can add
the gradient term
$\nabla f(x)\cdot \frac{z}{n}$ in the summation,
thanks to the property that
$\Phi_n(x,z)=\Phi_n(x,-z)$ for all $x\in n^{-1}\Z^d$, $z\in n^{-1}\Z_{\bf0}^d$ and $n\ge 1$;
and the inequality above follows from the Taylor formula and
the fact that $\sup_{n\ge1}\|\Phi_n\|_\infty<\infty$.
Hence by \eqref{a5-3-1-1} and the definition of the operator
$\bar \LL$
given by \eqref{l5-3-1a-2}, we can see that
$$
\lim_{n \rightarrow \infty}\sup_{x\in n^{-1}\Z^d: |x|\le R}|\bar {\LL}_{n,1}f(x)
-\bar \LL f(x)|=0.
$$
Therefore for any $R>1$,
$$
\lim_{n\rightarrow \infty}I_1^{n,R}=0.
$$

By \eqref{a5-2-2-2}, for any $R>1$,
\begin{align*}
\limsup_{n \rightarrow \infty}I_2^{n,R}
&\le 2\|f\|_\infty \limsup_{n \rightarrow \infty}\sup_{x\in \Z^d:
{|x|\le nR}} \frac{n^2}{\log(1+n)}\sum_{|z|\ge n}\frac{\kappa(x,z)}{|z|^{d+2}}=0.
\end{align*}
According to \eqref{l5-2-1-2},
$$\lim_{R\to\infty}\limsup_{n\to\infty} I_3^{n,R}=0.
$$
By the definition of the operator $\bar\LL$ given by \eqref{l5-3-1a-2} again, it is obvious that
$$\lim_{R\rightarrow \infty}I_4^{R}=0.
$$
Therefore,  \eqref{l5-3-1-2} is a consequence of all estimates above,
by first letting $n \rightarrow \infty$ and then taking $R \rightarrow \infty$.

(3)
According to \eqref{a5-3-1-1}, $A:=(a_{ij})_{1\le i,j\le d}$
given in Assumption {\bf (C2)} is non-negative definite.
Note that the solution of the  martingale problem for $(\bar {\LL}, C^2_c(\R^d))$
defined by \eqref{l5-3-1a-2}
 is always unique, since it corresponds to
 Brownian motion with a deterministic
covariance matrix $A$.
With (1) and (2) at hand, we can follow the argument of
the proof of Theorem \ref{t5-1} to obtain
the desired assertion.
\end{proof}

\section{Random walks in balanced random environments}
\label{section3} \rm

\begin{proof}[Proof of Theorem {\rm \ref{t5-2}(i)}]
First, we claim that
Assumptions {\bf (A0)} and {\bf(A1)} imply
that for a.s.\ $\w\in \Omega$, Assumption {\bf (B1)} holds, and that for
every
integer
$R\geq 1$, rational constant $\e>0$
 and $f\in C_c^2(\R^d)$,
\begin{equation}\label{t1-1-1}
\lim_{n \rightarrow \infty}
n^\alpha\sup_{x\in \Z^d: \atop{|x|\le n R}}
\Big| \!\!\!\sum_{z\in \Z^d \atop{ n\e<|z|<n/\e}}\!\!\!
\left(f\Big(\frac{x+z}{n}\Big)-f\left(\frac{x}{n}\right)\right)
\left(\frac{\kappa(x,z)-
\Ee [\kappa(x, z)]}{|z|^{d+\alpha}}\right) \Big|=0.
\end{equation}

The proof is mainly based on that of \cite[Proposition 5.6]{CKW}.
For the convenience of the reader, we
give
the details here.
Set $J(x,z):=\Ee[\kappa(x,z)]$. For $x\in \Z^d$, $R,\delta,\varepsilon>0$ and $h:\Z^d\times \Z^d \rightarrow \R$, define
\begin{align*}
q_1(x,\delta,h,\varepsilon)&:=\Pp\Big(\Big|\sum_{{z \in \Z^d:}
\atop{n\varepsilon \le |z|\le n/\varepsilon}}h(x,z)\frac{(\kappa(x,z)-J(x,z))}{|z|^{d+\alpha}}\Big|>\delta (n\varepsilon)^{-\alpha}\Big),\\
q_2(x,R,\delta)&:=\Pp\Big(\Big|\sum_{z \in \Z_{\bf0}^d:|z|\le R}\big(\kappa(x,z)-J(x,z)\big)\Big|>\delta R^{d}\Big),\\
q_3(x,R,\delta)&:=\Pp\Big(\Big|
\sum_{z\in \Z_{\bf0}^d: |z|\le R}\frac{(\kappa(x,z)-J(x,z))}{|z|^{d+\alpha-2}}\Big|>\delta R^{2-\alpha}\Big).
\end{align*}

Note that for a series of independent random variables $\{\eta_i\}_{1\le i\le n}$ with
$\Ee[\eta_i]=0$ for all $1\le i \le n$ and
$M:=\sup_{1\le i \le n}\Ee[|\eta_i|^{q}]<\infty$
for some $q\ge 1$,
by the Burkholder-Davis-Gundy inequality,
\begin{equation}\label{t1-1-2}
\Ee\left[\left|\sum_{i=1}^n \eta_i\right|^{q}\right]\le c_0
\Ee\left[\left(\sum_{i=1}^n \eta_i^2 \right)^{q/2}\right]\le c_1\max\{ n^{q/2-1},1 \} \sum_{i=1}^n \Ee\left[|\eta_i|^q\right]
\leq c_2 n^{\max \{ q/2,1 \}} M,
\end{equation}
where $c_0,c_1,c_2$ are positive constants that depend only on $q$.

For every $m \in \R_+$, let
\begin{align*}
S_m(i):=\Ee\left[\left(\sum_{z\in
\Z^d_{\bf0} : |z|\le 2^i}
\frac{\big(\kappa(x,z)-J(x,z)\big)}{|z|^{d+\alpha-2}}\right)^{m}\right]
=2^m
\Ee\left[\left(\sum_{j=1}^i \xi(j)\right)^{m}\right],
\end{align*}
where
$$
\xi(j)=\sum_{z \in \Z^d_{+, *}:
2^{j-1}<|z|\le 2^j}
\frac{ (\kappa(x,z)-J(x,z) )}{|z|^{d+\alpha-2}}.
$$
Recall that $J(x, z)=\Ee[\kappa (x, z)]$ and $\{\kappa(x,z): x\in \Z^d, z\in \Z^d_{+,*}\}$  are independent.
For every $m\in [2, p]$, by \eqref{t1-1-2}, there is a constant $c_3>0$ depending only on $m$ and $d$ so that
\begin{align*}
  \Ee\left[\left|  \sum_{z \in \Z^d_{+, *}: |z|=k}
\frac{\kappa(x,z)-J(x,z)}{|z|^{d+\alpha-2}}\right|^{m}\right]
&\leq   c_3k^{-m (d+\alpha-2)}  k^{\frac{(d-1)m}{2}}  \sup_{x\in \Z^d, z\in \Z^d_{\bf0}} \Ee [ \kappa (x, z)^{m}]\\
 &=  c_3 k^{\frac{m (3-2\alpha -d)}{2}}\sup_{x\in \Z^d,z\in \Z^d_{\bf0}} \Ee [ \kappa (x, z)^{m}].
\end{align*}
Hence by \eqref{t1-1-2} again,
\begin{equation}\label{e:note1} \begin{split}
\Ee[|\xi(j)|^{m}]
&=
\Ee\left[\left|\sum_{k=2^{j-1}+1}^{2^j}
\sum_{z \in \Z^d_{+, *}: |z|=k}\frac{\kappa(x,z)-J(x,z)}{|z|^{d+\alpha-2}}\right|^{m}\right]
  \\
&\le  c_4 2^{\frac{jm}{2}} 2^{\frac{jm (3-2\alpha -d)}{2}}\sup_{x\in \Z^d, z\in \Z_{\bf0}^d} \Ee [ \kappa (x, z)^{m}]
=  c_4 2^{\frac{jm(4-d-2\alpha)}{2}}\sup_{x\in \Z^d, z\in \Z_{\bf0}^d}\Ee [ \kappa (x, z)^{m}],
\end{split}\end{equation}
where $c_4>0$ is a constant that depends only on $m$ and $d$.
Consequently, according to the first inequality in \eqref{t1-1-2} and the H\"older inequality as well as Assumption {\bf(A1)}(i), we know that for each
$\e\in (0,d-4+2\alpha)$ (thanks to the assumption that $d>4-2\alpha$)
\begin{align*}
\sup_{i\ge 0}|S_p(i)|&\le c_5\sup_{i\ge 0}\Ee\left[\left|\sum_{j=1}^i \xi(j)^2\right|^{p/2}\right]\le c_5\sup_{i\ge 0}\Ee\left[\left(\sum_{j=1}^i 2^{\frac{\e jp}{2}}|\xi(j)|^p\right)
\left(\sum_{j=1}^i 2^{-\frac{\e jp}{p-2}}\right)^{p/2-1}\right]\\
&\le c_6\sup_{i\ge 0}\sum_{j=1}^i 2^{\frac{\e jp}{2}}\Ee[|\xi(j)|^p]\le
c_7\sum_{j=1}^\infty 2^{\frac{jp(4+\e-d-2\alpha)}{2}}<\infty.
\end{align*}
Then,  using the Markov inequality and the fact that $p >\frac{d+1}{2-\alpha}$,
we can find
a constant $\theta\in (0,1)$ such that
\begin{equation}\label{t1-1-3}
\sum_{R=1}^\infty \sum_{x\in B({\bf0},2R)\cap \Z^d}
q_3(x,R^\theta,\delta)\le c_{8}(\delta)\sum_{R=1}^\infty R^{d-(2-\alpha)\theta p}<\infty.
\end{equation}

Similarly,
since $p>\frac{2(d+1)}{d}$,
we can also show that
\begin{equation}\label{t1-1-4}
\sum_{R=1}^{\infty}\sum_{x \in
B({\bf0},2R)\cap \Z^d}\sum_{r=R^{\theta}/2}^{\infty}q_2(x,r,\delta)\le c_9(\delta)\sum_{R=1}^\infty R^d\sum_{r=R^\theta}^\infty
r^{-{dp}/{2}}<\infty.
\end{equation}
Thus, according to \eqref{t1-1-3}, \eqref{t1-1-4} and the Borel-Cantelli lemma,  for a.s. $\omega \in \Omega$ there exists $R_0(\w)\ge 2$ such that for all
$R>R_0(\w)$, $x\in B({\bf0},2R)\cap \Z^d$ and $r\ge R^{\theta}/2$,
$$
\sum_{z\in \Z^d_{\bf0}: |z|\le R^\theta}
\frac{\kappa(x,z)}{|z|^{d+\alpha-2}}
\le c_{10}R^{\theta(2-\alpha)} \quad \hbox{and} \quad
\sum_{z\in \Z^d_{\bf0} :|z|\le r } \kappa(x,z)
\le c_{11} r^{d}.
$$
Therefore, for $R\geq R_0(\omega)$,  every $r\in [R^\theta,  R]$ and $x\in B(0,2R)$,
\begin{align*}
\sum_{1\le |z|\le r}\frac{\kappa(x,z)}{|z|^{d+\alpha-2}}
&\le  \sum_{1\le |z|\le R^\theta}\frac{\kappa(x,z)}{|z|^{d+\alpha-2}}+
\sum_{R^\theta\le |z|\le r}\frac{\kappa(x,z)}{|z|^{d+\alpha-2}}\le   c_{10} R^{\theta(2-\alpha)}+\sum_{i=[\frac{\theta \log R}{\log 2}]}^{[\frac{\log r}{\log 2}]}
\sum_{2^i\le |z|< 2^{i+1}}\frac{\kappa(x,z)}{|z|^{d+\alpha-2}} \\
&\le   c_{10} R^{\theta(2-\alpha)}+c_{11} \sum_{i=[\frac{\theta \log R}{\log 2}]}^{[\frac{\log r}{\log 2}]}2^{-i(d+\alpha-2)}2^{(i+1)d} \le c_{12} r^{2-\alpha},
\end{align*}
and
\begin{equation}\label{t1-1-7} \begin{split}
\sum_{|z|> r}\frac{\kappa(x,z)}{|z|^{d+\alpha}}
&\leq  \sum_{i=[\frac{\log r}{\log 2}]}^\infty \sum_{ 2^i\leq |z|< 2^{i+1}}
\frac{\kappa(x,z)}{|z|^{d+\alpha}}\leq  \sum_{i=[\frac{\log r}{\log 2}]}^\infty 2^{-i (d+\alpha )} \sum_{z\in \Z^d_{\bf0}: |z|<2^{i+1}}
\kappa (x, z) \\
&\leq   c_{11} \sum_{i=[\frac{\log r}{\log 2}]}^\infty 2^{-i (d+\alpha )}  2^{(i+1)d}  \leq   c_{12} r^{-\alpha}.
\end{split}\end{equation}
This shows that
Assumption {\bf (B1)} holds for a.s. $\omega\in \Omega$.

For any fixed $f\in C_c^2(\R^d)$ and $n\ge1$, let
$$
f_n(x,z):=f \left(\frac{x+z}{n} \right)-f\left(\frac{x}{n}\right).
$$
 Then, for any $R\ge1$ and $\varepsilon,\delta>0$,
\begin{align*}
\sum_{n=1}^\infty \sum_{x\in B({\bf0},nR)\cap \Z^d}q_1(x,\delta,f_n ,\varepsilon)
&\le \sum_{n=1}^\infty \sum_{x\in B({\bf0},nR)\cap \Z^d} \delta^{-p}(\varepsilon n)^{\alpha p} (n\varepsilon)^{-p(d+\alpha)}\\
&\quad\times
 \Ee\left[\left|\sum_{z\in \Z^d: n\varepsilon\le |z|\le n/\varepsilon}f_n(x,z) (n\varepsilon)^{d+\alpha}\frac{(\kappa(x,z)-J(x,z))}{|z|^{d+\alpha}}\right|^{p}\right]\\
&\le c_{13} \delta^{-p} (\eps n)^{-pd} (nR)^d (n/\eps)^{{dp}/{2}} \| f_n\|_\infty^{p}
\sup_{x\in \Z^d, z\in \Z^d_{\bf0}} \Ee \left[ \kappa (x, z)^{p}\right] \\
&\le c_{14} \delta^{-p} \eps^{-\frac{3pd}{2}} R^d \| f\|_\infty^{p} \,
 \sum_{n=1}^\infty n^{-(\frac{p}{2}-1)d}<\infty.
\end{align*}
Here the first inequality is due to the Markov inequality, the second inequality follows
from \eqref{balacd}, \eqref{t1-1-2} and Assumption {\bf (A1)}(i),
 and in the last inequality
we have used the fact that
$p>\frac{2(d+1)}{d}$.
By the Borel-Cantelli lemma, there is a subset $\Omega_f$
of $\Omega$ of full probability  (which may depend on $f$)
so that for every $\omega\in \Omega_f$ and every positive rational constants
 $\eps,  \delta$ and
 integer $R\ge1$,
 there exists $N_0:=N_0(\omega,R,\delta,\varepsilon,\|f\|_\infty)\ge1$
such that for every $n>N_0$ and $\w\in \Omega_f$,
$$
n^\alpha \sup_{x\in B({\bf0},nR)\cap \Z^d} \Big|\sum_{{z \in \Z^d:}
\atop{n\varepsilon \le |z|\le n/\varepsilon}}f_n(x,z)\frac{(\kappa(x,z)-J(x,z))}{|z|^{d+\alpha}}\Big|\le \delta \varepsilon^{-\alpha}.
$$
Taking $\delta \to 0$ in the inequality above,
we   can obtain that for every given $f\in C_c^2(\R^d)$, \eqref{t1-1-1} holds for each
$\w\in \Omega_f$.

Now we are going to show that \eqref{t1-1-1} holds
on some subset $\Omega_0$ of $\Omega$ having full probability that is independent of $f$.

Let $\Upsilon\subset C_c^2(\R^d)$ be a countable dense subset
in $(C_c(\R^d),\|\cdot\|_\infty)$, and set $\Omega_1:
=\bigcap_{f\in \Upsilon}\Omega_f$.
Then $\Omega_1$ is of full probability.

Given some $R>0$ and $\e>0$, define (for simplicity, we omit the parameter $\w$ in $T_n(f)$)
\begin{align*}
T_n(f) =n^\alpha\sup_{x\in \Z^d:|x|\le nR}Q_n(f,x),\quad f\in C_c^2(\R^d),
\end{align*}
where
\begin{align*}
Q_n(f,x) =\Big|\sum_{z\in \Z^d: n\e\le |z|\le n\e^{-1}}
\Big(f\Big(\frac{x+z}{n}\Big)-f\Big(\frac{x}{n}\Big)\Big)\frac{\kappa(x,z)-J(x,z)}{|z|^{d+\alpha}}\Big|.
\end{align*}
Therefore, for every $f,g\in C_c^2(\R^d)$,
\begin{align*}
 |T_n(f)-T_n(g)|
&\le n^\alpha \sup_{x\in \Z^d:|x|\le nR}|Q_n(f,x)-Q_n(g,x)|\\
&\le n^\alpha\sup_{x\in \Z^d:|x|\le nR}\Bigg|\sum_{z\in \Z^d:n\e\le |z|\le n\e^{-1}}
\left[(f-g)\Big(\frac{x+z}{n}\Big)-(f-g)\Big(\frac{x}{n}\Big)
\right] \frac{\kappa(x,z)-J(x,z)}{|z|^{d+\alpha}}\Bigg|\\
&\le 2n^\alpha\|f-g\|_\infty\bigg(\sup_{x\in \Z^d:|x|\le nR}\sum_{|z|\ge n\e}\frac{\kappa(x,z)}{|z|^{d+\alpha}}+
\sup_{x\in \Z^d:|x|\le nR}\sum_{|z|\ge n\e}\frac{J(x,z)}{|z|^{d+\alpha}}\bigg).
\end{align*}
By \eqref{t1-1-7},
there exists a subset $\Omega_2\subset \Omega$ having
full probability such that for all
$\w\in \Omega_1$,
 integer  $R\geq 1$,
 rational constant $\e>0$
 and $n>N_0(\w)$ large enough,
\begin{align*}
\sup_{x\in \Z^d:|x|\le nR}\sum_{|z|>n\e}\frac{\kappa(x,z)}{|z|^{d+\alpha}}\le c_{15} (n\e)^{-\alpha}.
\end{align*}
Meanwhile, since $J(x,z)$ is uniformly bounded,
\begin{align*}
\sup_{x\in \Z^d:|x|\le nR}\sum_{|z|>n\e}\frac{J(x,z)}{|z|^{d+\alpha}}\le c_{16}(n\e)^{-\alpha}.
\end{align*}
Combining all the estimates above together yields that for all $f,g\in C_c^\infty(\R^d)$ and $\w\in \Omega_1$,
$$
\limsup_{n \rightarrow \infty}|T_n(f)-T_n(g)|\le c_{17}(\e)\|f-g\|_\infty,
$$
where $c_{17}$ is independent of $f,g$ and $n$.

For any $f\in C_c^\infty(\R^d)$, we can find a sequence $\{f_k\}_{k=1}^\infty\subset \Upsilon$ such that
$\lim_{k \rightarrow \infty}\|f_k-f\|_\infty=0$.
Let $\Omega_0:=\Omega_1\cap \Omega_2$.
Obviously
$\Omega_0$ has
full probability. For every $\w\in
\Omega_0$,
\begin{align*}
\lim_{n \rightarrow \infty}T_n(f)&\le \lim_{n \rightarrow \infty}T_n(f_k)+
\limsup_{n \rightarrow \infty}|T_n(f)-T_n(f_k)|\le
c_{17} (\omega) \|f-f_k\|_\infty,
\end{align*}
where in the second inequality we used the fact that \eqref{t1-1-1} holds for every
$\w \in \Omega_0$ and every $g\in \Upsilon$.
Then, letting $k \rightarrow \infty$ in the inequality above, we know that \eqref{t1-1-1} is true
for every $\w\in \Omega$ and $f\in C_c^2(\R^d)$.
Therefore, by  \eqref{t1-1-1} and Assumption {\bf (A1)}(ii) (in particular, \eqref{e:1.3}), we
conclude that
$$\lim_{n \rightarrow \infty}
n^\alpha\sup_{x\in \Z^d: \atop{ |x|\le nR}}
\Bigg| \sum_{z\in \Z^d:\atop {n\e<|z|<n/\e} }
\left(f\Big(\frac{x+z}{n}\Big)-f\left(\frac{x}{n}\right)\right)
\left(\frac{\kappa(x,z)-K(\frac{x}{n},\frac{z}{n})}{|z|^{d+\alpha}}\right) \Bigg| =0$$
holds for  every  $\w\in \Omega_0$,   rational constant  $\e>0$ small enough,
 integer $R\geq 1$
 large enough
 and any $f\in C_c^2(\R^d)$.   In particular, this implies that
 \eqref{a5-3-1} holds for  every  $\w\in \Omega_0$,
  integer $R\geq 1$
 large enough   and any $f\in C_c^2(\R^d)$.   Hence, Assumption {\bf (B2)} holds for
every $\w \in \Omega_0$. The conclusion of Theorem \ref{t5-2}(i) now
 follows from Theorem \ref{t5-1}.
\end{proof}

\begin{proof}[Proof of Theorem {\rm \ref{t5-2}(ii)}] Suppose that Assumptions {\bf(A0)} and {\bf(A2)} hold.
We will show that
Assumptions {\bf(C1)} and {\bf(C2)} are satisfied for a.s. $\omega\in \Omega$.
 For any $x\in \Z^d$ and $z\in \Z_{\bf0}^d$, set $J(x,z):=\Ee[\kappa(x,z)]$. We set for $x\in \Z^d$, $R,\delta>0$ and $h:\Z^d\times \Z^d \rightarrow \R$,
\begin{align*}
q_4(x,\delta,h)&:=\Pp\Big(\Big|\sum_{{z \in \Z^d:}
{1\le |z|\le n}}h(x,z)\frac{(\kappa(x,z)-J(x,z))}{|z|^{d+2}}\Big|>\delta\frac{\log(n+1)}{n^2} \Big),\\
q_5(x,R,\delta)&:=\Pp\Big(\Big|
\sum_{z\in \Z_{\bf0}^d: |z|\le R}\frac{(\kappa(x,z)-J(x,z))}{|z|^{d}}\Big|>\delta \log(R+1)\Big).
\end{align*}

For any fixed $f\in C_c^2(\R^d)$ and $n\ge1$, let $f_n(x,z):=f (\frac{x+z}{n} )-f\left(\frac{x}{n}\right)$. Then for any $R,n\ge1$, $x\in B({\bf0},nR)\cap  \Z^d$ and $\delta>0$,
\begin{equation} \label{e:3.70}\begin{split}
& q_4(x,\delta,f_n )  \\
&\leq   \Pp\Big(\Big|\sum_{{z \in \Z^d_{+,*}:}
{1\le |z|\le n}}f_n(x,z)\frac{(\kappa(x,z)-J(x,z))}{|z|^{d+2}}\Big|>\frac{\delta}{2}\frac{\log(1+n)}{n^2} \Big)   \\
&\quad   +\Pp\Big(\Big|\sum_{{z \in \Z^d_{+,*}:}
{1\le |z|\le n}}f_n(x,-z)\frac{(\kappa(x,z)-J(x,z))}{|z|^{d+2}}\Big|>\frac{\delta}{2}\frac{\log(1+n)}{n^2} \Big)   \\
&\leq    \Pp\Big(\Big |\!\sum_{{z \in \Z^d_{+,*}:}
{1\le |z|\le n }}\!\!\left\langle \nabla^2 f(x_*(z,n)),\frac{z}{|z|}\otimes \frac{z}{|z|}\right\rangle\frac{(\kappa(x,z)-J(x,z))}{|z|^{d}}\Big|\!>\delta{\log(n+1)}
\Big) \\
&\quad  + \Pp\Big(\Big |\!\sum_{{z \in \Z^d_{+,*}:}
{1\le |z|\le n }}\!\!\left\langle \nabla^2 f(x_*(-z,n)),\frac{z}{|z|}\otimes \frac{z}{|z|}\right\rangle\frac{(\kappa(x,z)-J(x,z))}{|z|^{d}}\Big|\!>\delta{\log(n+1)}
\Big).
\end{split}\end{equation}
Here in the first inequality we used the balanced condition \eqref{balacd},  in the second inequality we applied the Taylor formula and $x_*(z,n)=x/n+ \theta_0(x,z,n)z/n$ with $\theta_0(x,z,n)\in [0,1]$ being a constant
depending only
on $x,z$ and $n$.
By \eqref{l5-3-1b} in Assumption {\bf(A2)}(i), for
$c_0:=c_*/(2\|\nabla^2 f\|_\infty)^\eta$,
$$
C_*:=
\sup_{x\in \Z^d, z\in \Z^d_{\bf0},n\ge1}\Ee e^{c_0| \langle \nabla^2f(x_*(z,n)),\frac{z}{|z|}\otimes \frac{z}{|z|}\rangle (\kappa(x,z)-J(x,z))|^\eta} <\infty.
$$
This along with the Markov inequality implies that there exist constants $c_1,c_2>0$ such that for all $x\in \Z^d$, $z\in \Z^d_{\bf 0}$, $n\ge1$ and $t>0$,
\begin{equation} \label{e:3.8}\begin{split}
&  \Ee\left[ \exp\left(t\left| \left \langle \nabla^2 f(x_*(z,n)),\frac{z}{|z|}\otimes \frac{z}{|z|}\right\rangle (\kappa(x,z)-J(x,z))\right|\right)\right]  \\
&=   1+ t\int_0^\infty e^{tr} \Pp\left(\left| \left\langle \nabla^2 f(x_*(z,n)),\frac{z}{|z|}\otimes \frac{z}{|z|}\right\rangle (\kappa(x,z)-J(x,z))\right|\ge r\right)\,dr
 \\
&\le   1+ C_*t\int_0^\infty e^{tr-c_0r^\eta}\,dr\le c_1 \exp\left(c_2{t^{\eta/(\eta-1)}}\right).
\end{split}\end{equation}
Hence by \eqref{e:3.70}, \eqref{e:3.8},
$\eta\in (1,2)$, the independence of $\{\kappa(x,z): x\in \Z^d, z\in \Z^d_{+,*}\}$ (due to Assumption {\bf(A0)}) and  \cite[(2.95) in Theorem 2.51, p.\ 45]{BDR},
we get
$$
q_4(x,\delta,f_n )\le c_3\exp\left(-c_4\log^\eta(1+n)\right).
$$
 We emphasize that the constants $c_i$ $(i=1,\cdots,4)$ above are independent of $R$ and $n$ (but
may depend on
 $\|f\|_\infty$ and $\delta$).
Hence, for any $R\ge1$ and $\delta>0$,
\begin{equation}\label{e:keye}
\sum_{n=1}^\infty \sum_{x\in B({\bf0},nR)\cap  \Z^d}q_4(x,\delta,f_n)\le c_5\sum_{n=1}^\infty \exp(- c_6 \log^\eta(1+n))<\infty.
\end{equation}

Similarly,  we can prove that for any $\theta\in (0,1)$,
$$
\sum_{R=1}^\infty \sum_{x\in B({\bf0},2R)\cap \Z^d}\sum_{r=R^\theta}^R
q_5(x,r,\delta)\le c_{7}\sum_{R=1}^\infty \exp(-c_{8} \log^\eta(1+R))<\infty.
$$
By the Borel-Cantelli lemma, for a.s $\w \in \Omega$,
there exists $R_0(\w)\ge1$ such that for all
$R>R_0(\w)$, $x\in B({\bf0},2R)\cap \Z^d$ and $r\in [R^\theta,R]$,
\begin{equation}\label{t1-2-1}
\sum_{z\in \Z_{\bf0}^d: |z|\le r}
\frac{\kappa(x,z)}{|z|^{d}}
\le c_{9}\log (1+r).
\end{equation}
By the proof of \eqref{t1-1-4},
we can show that for a.s. $\omega \in \Omega$ there exists $R_0(\w)\ge1$ such that for all
$R>R_0(\w)$, $x\in B({\bf0},2R)\cap \Z^d$ and $r\ge R^{\theta}/2$,
$$
\sum_{z\in \Z_{\bf0}^d:|z|\le r} \kappa(x,z) \le c_{10} r^{d},
$$
which along with the same argument for \eqref{t1-1-7} gives us that \eqref{a5-2-2-2} holds for a.s. $\omega\in \Omega$. Thus, we know that
Assumption {\bf (C1)} holds for a.s. $\omega\in \Omega$.

\smallskip

To show that
Assumption {\bf(C2)} holds for a.s. $\omega\in \Omega$
as well, we return to \eqref{e:keye}. By
the Borel-Cantelli lemma, there is a set $\Omega_f \subset \Omega$ of full probability such that for every $\omega \in \Omega_f$
and for every rational
constant $\delta$ and integer $R\ge1$,
 there exists an integer $N_0:=N_0(\omega,R,\delta, \|f\|_\infty)\ge 2$
such that for all $n>N_0$ and $x\in B({\bf0},nR)\cap  \Z^d$,
$$
\frac{n^2}{\log(n+1)} \Big|\sum_{{z \in \Z^d:}
{1\le |z|\le n}}f_n(x,z)\frac{\kappa(x,z)-J(x,z)}{|z|^{d+2}}\Big|\le \delta .
$$
Taking $\delta \to 0$ in the inequality above,
we can obtain that for every $\w \in \Omega_f$ and every
integer  $R\geq 1$,
\begin{equation}\label{e:3.7}
\lim_{n \rightarrow \infty}
\frac{n^2}{\log (1+n)}\sup_{x\in \Z^d: \atop{|x|\le n R}}
\Big| \!\!\!\sum_{z\in \Z^d:{ 1\le |z|\le n }}\!\!\!
 f_n(x, z)
\frac{\kappa(x,z)-
\Ee [\kappa(x, z)]}{|z|^{d+2}} \Big|=0.
\end{equation}

It is known that there is a countable subset $\Upsilon$ of $C^2_c(\R^d)$ that is dense in
$(C^2_c(\R^d), \|\cdot\|_{2,\infty})$ (with $\|f\|_{2,\infty}:=\| f \|_\infty+\|\nabla f\|_\infty+\|\nabla^2 f\|_\infty$)\footnote{For reader's convenience, we present a proof of this fact in this footnote. Denote by $C_\infty(\R^d)$ the space of continuous functions on $\R^d$ that vanish at infinity, and $C_\infty^2 (\R^d)$ the space of $C^2$-smooth functions on $\R^d$ that
 together with their derivatives up to second orders  vanish at infinity.
  Observe that $C_\infty (\R^d)$ is a separable Banach space under the uniform norm $\|\cdot\|_\infty$.
 Denote by $\Phi$ the
isometric
map from $(C_\infty^2 (\R^d), \| \cdot \|_{2, \infty})$ into
$(C_\infty (\R^d)^{1+d+d^2}, \| \cdot \|_\infty)$ defined by $\Phi (f)=(f, \nabla f, \nabla^2 f)$.
Note that  $\Phi (C_\infty^2 (\R^d))$ is a closed subspace of
$(C_\infty (\R^d)^{1+d+d^2}, \| \cdot \|_\infty)$
and hence is separable under the norm $\| \cdot \|_\infty$.
It follows that $(C^2_\infty (\R^d), \| \cdot \|_{2, \infty})$ is separable. Let $\{g_n; n\geq 1\}$ be a countable dense subsequence in $(C^2_\infty (\R^d), \| \cdot \|_{2, \infty})$, and
$\{\varphi_k; k\geq 1\}$ a sequence of smooth functions with compact support on $\R^d$
so that $\varphi_k (x)=1$ for $|x|\leq k$ and $\varphi_k (x)=0$ for $|x|\geq k+1$ with
$\sup_{k\geq 1} \| \varphi_k\|_{2, \infty} <\infty$.
Then $\{g_n\varphi_k; n, k\geq 1\}$ is a countable dense sequence in $(C^2_c (\R^d),
\| \cdot \|_{2,\infty})$.
}.
Set $\Omega_1:=\bigcap_{f\in \Upsilon}\Omega_f$.
Clearly, $\Omega_1$ has full probability.

Given some $R>0$, define (for simplicity, we omit the parameter $\w$ in $T_n(f)$)
\begin{align*}
T_n(f) = \frac{n^2}{\log(1+n)}\sup_{x\in \Z^d:|x|\le nR}Q_n(f,x),\quad f\in C_c^2(\R^d),
\end{align*}
where
\begin{align*}
Q_n(f,x) =\Big|\sum_{z\in \Z^d: 1\le |z|\le n}
\Big(f\Big(\frac{x+z}{n}\Big)-f\Big(\frac{x}{n}\Big)\Big)\frac{\kappa(x,z)-J(x,z)}{|z|^{d+2}}\Big|.
\end{align*}
Therefore, for every $f,g\in C_c^2(\R^d)$,
\begin{align*}
 &|T_n(f)-T_n(g)|\\
 &\le \frac{n^2}{\log(1+n)} \sup_{x\in \Z^d:|x|\le nR}|Q_n(f,x)-Q_n(g,x)|\\
&\le \frac{n^2}{\log(1+n)}\sup_{x\in \Z^d:|x|\le nR}\Bigg|\sum_{z\in \Z^d: 1\le |z|\le n }
 \left( (f-g)\Big(\frac{x+z}{n}\Big) - (f-g)\Big(\frac{x}{n}\Big)
\right) \frac{\kappa(x,z)-J(x,z)}{|z|^{d+2}}\Bigg|\\
&\le \frac{1}{2\log(1+n)}\|\nabla^2 (f-g)\|_\infty\bigg(\sup_{x\in \Z^d:|x|\le nR}\sum_{1\le |z|\le n}\frac{\kappa(x,z)}{|z|^{d+2}}+
\sup_{x\in \Z^d:|x|\le nR}\sum_{1\le |z|\le n}\frac{J(x,z)}{|z|^{d+2}}\bigg),
\end{align*}where in the last inequality we used \eqref{balacd} and the mean value theorem.
By \eqref{t1-2-1} we know there exists a set $\Omega_2$ of
 full probability such that for all $\w\in \Omega_2$, any
 integer  $R\geq 1$
 and $n>N_0(\w)$ large enough,
\begin{align*}
\sup_{x\in \Z^d:|x|\le nR}\sum_{|z|\le n}\frac{\kappa(x,z)}{|z|^{d+2}}\le c_{11} \log(1+n).
\end{align*}
Meanwhile, since $J(x,z)$ is uniformly bounded,
\begin{align*}
\sup_{x\in \Z^d:|x|\le nR}\sum_{1\le |z|\le n}\frac{J(x,z)}{|z|^{d+2}}\le c_{12}\log(1+n).
\end{align*}
Combining all the estimates above together yields that for all $f,g\in C_c^2(\R^d)$ and
$\w\in \Omega_2$,
$$
\limsup_{n \rightarrow \infty}|T_n(f)-T_n(g)|\le c_{13}(\e)\|\nabla^2(f-g)\|_\infty,
$$
where $c_{13}$ is independent of $f,g$ and $n$.
With this assertion at hand, we can then follow the proof of Theorem {\rm \ref{t5-2}(i)} to deduce
that \eqref{e:3.7} holds for all $f\in C_c^2(\R^d)$ and
 every $\w\in \Omega_0:=\Omega_1 \cap \Omega_2$. This implies
 that \eqref{a5-3-1-2} holds for every $\w \in \Omega_0$
with
$\Phi_n(x,y):=J(nx,nz)$ for each $x\in n^{-1}\Z^d$, $z\in n^{-1}\Z_{\bf0}^d$.
Clearly
\eqref{matix} implies that \eqref{a5-3-1-1} holds for the function $
\Phi_n(x,z)$ defined above.
It is also obvious that $\Phi_n(x,z)=\Phi_n(x,-z)$ for $x\in n^{-1}\Z^d$, $z\in n^{-1}\Z_{\bf0}^d$ and $n\ge 1$, and that $\sup_{n\ge 1}\|\Phi_n\|_\infty<\infty$ thanks to \eqref{l5-3-1b}.
Thus we have established that Assumption {\bf (C2)} holds for every $\w \in \Omega_0$.
Therefore by Theorem \ref{t5-1-1-2}, we
get the conclusion for $\alpha=2$ in Theorem \ref{t5-2}.
\end{proof}

Next, we present the

\begin{proof}[Proof of Proposition $\ref{t5-4}$]
By \eqref{t5-4-1} and the continuity of $K$ in $\R^d_{\bf0}$,
we have
\begin{equation}\label{t1-1-5}
K(sz)=K(z), \quad  z\in \R_{\bf0}^d,\ s>0.
\end{equation}
Therefore, the function $K:\R^d_{\bf0} \rightarrow \R_+$ can
be viewed as a function defined on  $\mathbb{S}^{d-1}$.

According to \eqref{t5-4-1}, for every
integer $R\geq 1$,
 \begin{equation}\label{e:llff}
\lim_{n \rightarrow \infty}\sup_{x\in n^{-1}\Z^d: |x|\le R}\frac{1}{\log (1+n)}\sum_{z\in \Z_{\bf0}^d: |z|\le n}
\Big|\frac{z_i z_j \Ee[\kappa(nx,z)]}{|z|^{d+2}}-
\frac{z_i z_j K({z}/{n})}{|z|^{d+2}}\Big|=0.\end{equation}
 Indeed, \eqref{t5-4-1} implies
that   for
any integer
$R\geq 1$ and $\varepsilon\in (0,1)$,
$$\lim_{n \rightarrow \infty}\sup_{x\in n^{-1}\Z^d: |x|\le R}\frac{1}{\log (1+n)}\sum_{z\in \Z_{\bf0}^d:
n^\varepsilon\le |z|\le n}
\Big|\frac{z_i z_j \Ee[\kappa(nx,z)]}{|z|^{d+2}}-
\frac{z_i z_j K({z}/{n})}{|z|^{d+2}}\Big|=0.$$On the other hand, it follows from the boundedness of $K$ and $\sup_{x\in \Z^d, z\in \Z^d_{\bf 0}} \Ee[\kappa(x,z)]<\infty$ that
 for
 any integer
 $R\geq 1 $ and $\varepsilon\in (0,1)$,
\begin{align*}&\sup_{x\in n^{-1}\Z^d: |x|\le R}\frac{1}{\log (1+n)}\sum_{z\in \Z_{\bf0}^d:  |z|\le
n^\varepsilon}
\Big|\frac{z_i z_j \Ee[\kappa(nx,z)]}{|z|^{d+2}}-
\frac{z_i z_j K({z}/{n})}{|z|^{d+2}}\Big|\\
&\le \frac{\sup_{x\in \Z^d, z\in \Z^d_{\bf 0}} \Ee[\kappa(x,z)]+\|K\|_\infty}{\log (1+n)}\sum_{z\in \Z_{\bf0}^d:  |z|\le  n^\varepsilon}\frac{|z_i|| z_j|}{|z|^{d+2}}\\
&\le c_0\frac{ \varepsilon \log(1+ n)}{\log(1+n)}\to  c_0\varepsilon\quad \hbox{ as  } n\to \infty.
\end{align*}
  Putting both estimates above together    and letting $\varepsilon\to0$,
we get
\eqref{e:llff}.
Hence, in order to
verify \eqref{matix} it suffices to show
\begin{align*}
\lim_{n \rightarrow \infty}
\Big|\frac{1}{\log (1+n)}\sum_{z\in \Z_{\bf0}^d: |z|\le n}
\frac{z_i z_j K({z}/{n})}{|z|^{d+2}}-a_{ij}\Big|=0,\quad 1\le i,j\le d,
\end{align*}
where $a_{ij}=\int_{\mathbb{S}^{d-1}}K(\theta)\theta_i \theta_j \,d\theta$.

It follows from \eqref{t1-1-5} that
\begin{equation}\label{t1-1-6}
\begin{split}
\frac{1}{\log (1+n)}\sum_{z\in \Z_{\bf0}^d
: |z|\le n}
\frac{z_i z_j K({z}/{n})}{|z|^{d+2}}&=
\frac{1}{n^d \log (1+n) }\sum_{z\in n^{-1}\Z^d: 0<|z|\le 1} \frac{z_i z_j K(z)}{|z|^{d+2}}\\
&=\frac{1}{\log (1+n)}\sum_{z\in n^{-1}\Z^d: 0<|z|\le 1}\int_{Q_n(z)}\frac{z_i z_j K(z)}{|z|^{d+2}}\,dy,
\end{split}
\end{equation}
where
$Q_n(z):=\Pi_{1\le i\le d} (z_i-1/2n,z_i+1/(2n)]$ for
 $z:=(z_1,\cdots,z_d)\in n^{-1}\Z^d$. By the mean value theorem and the fact that $K$
is uniformly bounded, for any $z\in n^{-1}\Z^d$ with $\frac{\sqrt{d}}{n}\le  |z|\le 1$ and any $y\in Q_n(z)$,
\begin{equation}\label{e:3.43}
\bigg|\frac{z_iz_jK(z)}{|z|^{d+2}}-\frac{y_iy_jK(y)}{|y|^{d+2}}\bigg|
\le c_1\bigg(\frac{1}{n(|z|-\frac{\sqrt{d}}{2n})^{d+1}}+\frac{\xi_{n}}{(|z|-\frac{\sqrt{d}}{2n})^d}\bigg),
\end{equation}
where
$$
\xi_{n}:=\sup_{0<|y|, |z|\le 1,|y-z|\le \frac{\sqrt{d}}{2n}}|K(z)-K(y)|.
$$
Therefore,
\begin{align*}
  I(n)  &:=
\bigg|\frac{1}{\log (1+n)}\sum_{z\in n^{-1}\Z^d: 0<|z|\le 1}\int_{Q_n(z)}\frac{z_i z_j K(z)}{|z|^{d+2}}\,dy-
\frac{1}{\log (1+n)}\int_{\{1/n\le |y|\le 1\}}
\frac{y_i y_j K(y)}{|y|^{d+2}}\,dy\bigg|\\
&\le \frac{1}{n^d \log (1+n)}\sum_{z\in n^{-1}\Z^d: |z|< \frac{\sqrt{d}}{n}}\bigg|\frac{z_i z_j K(z)}{|z|^{d+2}}\bigg|\\
&\quad+\frac{1}{\log (1+n)}\int_{\{(\frac{\sqrt{d}}{2n}\wedge \frac{1}{n})\le |y|\le \frac{3\sqrt{d}}{2n}\}\cup \{1\le |y|\le1
+\frac{\sqrt{d}}{2n}\}  }
\bigg|\frac{y_i y_j K(y)}{|y|^{d+2}}\bigg|\,dy\\
&\quad + \frac{1}{n^d\log (1+n)}
\sum_{z\in \Z^d: \frac{\sqrt{d}}{n}\le |z|\le 1}
\int_{Q_n(z)}\bigg|\frac{z_iz_jK(z)}{|z|^{d+2}}-\frac{y_iy_jK(y)}{|y|^{d+2}}\bigg|\,dy\\
&=:I_1^n+I_2^n+I_3^{n}.
\end{align*}
It is clear that
$$
I_1^n\le \frac{c_1}{\log (1+n)},\qquad I_2^n\le \frac{c_2}{\log (1+n)},
$$
and, by \eqref{e:3.43},
\begin{align*}
I_3^{n} &\le \frac{c_1}{n^d\log (1+n)}\sum_{k=[\sqrt{d}]}^{n}\sum_{z\in n^{-1}\Z^d:  |z|=\frac{k}{n}}
\bigg(\frac{n^d}{(k-\frac{\sqrt{d}}{2})^{d+1}}+\frac{n^d\xi_{n}}{(k-\frac{\sqrt{d}}{2})^d}\bigg)\le c_3\Big(\frac{1}{\log (1+n)}+\xi_{n}\Big).
\end{align*}
Since $K(\cdot)$ is a continuous function on $\R_{\bf0}^d$ and satisfies \eqref{t1-1-5},
$\lim_{n\rightarrow \infty}\xi_{n}=0$. Hence we deduce from the above estimates that
\begin{align*}
\lim_{n \rightarrow \infty}
 I(n) =0.
\end{align*}
This along with \eqref{t1-1-6} gives us
\begin{align*}
&\lim_{n \rightarrow \infty}\bigg|\frac{1}{\log (1+n)}\sum_{z\in \Z_{\bf0}^d: |z|\le n} \frac{z_i z_j K({z}/{n})}{|z|^{d+2}}
-\int_{\mathbb{S}^{d-1}}\theta_i \theta_j K(\theta)\,d\theta\bigg|\\
&=\lim_{n \rightarrow \infty} \bigg|\frac{1}{\log (1+n)}\int_{\{1/n\le |y|\le 1\}}
\frac{y_i y_j K(y)}{|y|^{d+2}}\,dy-\int_{\mathbb{S}^{d-1}}\theta_i \theta_j K(\theta)\,d\theta\bigg|\\
&=\lim_{n \rightarrow \infty}\bigg|\frac{1}{\log (1+n)}\int_{1/n}^1 r^{-1}
\int_{\mathbb{S}^{d-1}}\theta_i \theta_j K(\theta)\,d\theta\,dr-
\int_{\mathbb{S}^{d-1}}\theta_i \theta_j K(\theta)\,d\theta\bigg|\\
&=0,
\end{align*}
where in the second
equality we used  \eqref{t1-1-5}.
In particular, we have verified that the limit in \eqref{matix} exists with $a_{ij}=
\int_{\mathbb{S}^{d-1}}\theta_i \theta_j K(\theta)\,d\theta$.
\end{proof}

\section{Extensions and remarks}\label{section4}

\subsection{Extensions}

From the proof for the case $\alpha\in(0,2)$ in Theorem \ref{t5-2},
one can see that, at the expense of a
higher moments condition on $\kappa(x,z)$,
 it is possible to relax the independence
 Assumption {\bf(A0)} to a block independence condition that
 $$
\left\{\big\{\kappa(x,z):  z\in \Z_{+,*}^d\,\, \mathrm{ with}\,\, |z|=
r \big\}: x\in \Z^d,
r\ge1\right\}
\hbox{ are independent}.
$$

\medskip

When $\alpha\in (0,1)$, since $1/|z|^{d+\alpha}$ is integrable over the unit ball $B(0, 1)$ in $\R^d$,
 we can drop the balanced condition \eqref{balacd} on $\kappa (x, z)$.
 In this case,
Theorem $\ref{t5-2}$(i) holds
with Assumption {\bf(A0)} and Assumption {\bf(A1)}(i)  being replaced by

\medskip

\noindent{\bf Assumption {\bf(A0$^*$)}}\,\, \emph{$\{\kappa(x,z): x \in \Z^d,
z\in \Z^d_{\bf0} \}$
is a sequence of independent non-negative
random variables.}

\medskip

\noindent{\bf Assumption {\bf(A1$^*$)}}(i)\,\,
\emph{$d>2-2\alpha$ and there is some
$p>\max\{\frac{2(d+1)}{d},
\frac{d+1}{1-\alpha}\}$ such that
$$
\sup_{x\in \Z^d, z\in \Z^d_{\bf0}}\Ee \left[ \kappa(x,z)^p \right]<\infty .
$$
}

Indeed,
as is shown in
the proof of \cite[Proposition 5.6]{CKW},
we know that
Assumption {\bf (A0$^*$)} and {\bf(A1$^*$)}(i) together imply
Assumption {\bf (B1${}^*$)}.
Hence the conclusion of Theorem \ref{t5-2}(i) still holds by following the proof of Theorem \ref{t5-2}(i) and by using Theorem \ref{t5-1*}.

\subsection{Remarks}\label{e:diff}

Now suppose $\alpha>2$. We consider the following two conditions.

\medskip

{\paragraph{{\bf Assumption (D1)}}\it
$$\sup_{x\in \Z^d}\sum_{z\in \Z_{\bf0}^d}\frac{\kappa(x,z)}{|z|^{d+\alpha-2}}<\infty.$$}

\medskip

{\paragraph{{\bf Assumption (D2)}} \it
 For any $n\ge1$, there exists a
function
$\Phi_n(x,z):  n^{-1}\Z^d \times
n^{-1}\Z_{\bf0}^d\rightarrow (0,\infty)
$
with $\Phi_n(x,z)=\Phi_n(x,-z)$ for all $x \in
n^{-1}\Z^d$ and $z\in
n^{-1}\Z_{\bf0}^d$,  so that $\sup_{n\ge 1}\|\Phi_n\|_\infty<\infty$, and for any
integer $R\geq 1$
and every $f\in C_c^2(\R^d)$,
$$
\lim_{n \rightarrow \infty}
n^2\sup_{x\in \Z^d: \atop{ |x|\le nR}}
\Bigg| \sum_{z\in \Z_{\bf0}^d  }
\!\!\left(f\Big(\frac{x+z}{n}\Big)-f\left(\frac{x}{n}\right)\!\right)
\!\!\left(\frac{\kappa(x,z)- \Phi_n({x}/{n},{z}/{n})}{|z|^{d+\alpha}}\!\right) \Bigg| =0,
$$
and that for any
integer $R\ge1$ and $1\le i,j\le d$,
$$
\lim_{n\to\infty}\sup_{x\in n^{-1}\Z^d: |x|\le R}\Bigg|
\sum_{z\in \Z_{\bf0}^d} {z_iz_j  \Phi_n(x, {z}/{n})}/{|z|^{d+2}}-a_{ij}\Bigg|=0
$$
for some constant matrix $A:=(a_{ij})_{1\le i,j\le d}$ on $\R^d$.}

\bigskip

We can establish the following theorem by a similar argument as
that
of Theorem \ref{t5-1-1-2}.

\medskip

\begin{theorem} Let $\alpha>2$. Assume that
the balanced condition \eqref{balacd}, and Assumptions {\bf(D1)} and {\bf(D2)} hold.
Then the conclusion of Theorem $\ref{t5-1-1-2}$ holds.
\end{theorem}

However, we
are unable obtain the corresponding result of  Theorem \ref{t5-2} for $\alpha>2$
in random environments
under the Assumption {\bf(A0)} and \eqref{matix} (but with  $|z|^{-d-\alpha}$ in place of  $|z|^{-d-2}$ and without the $(\log n)^{-1}$ factor). It appears
that the Borel-Cantelli argument to verify Assumptions {\bf (D1)} and {\bf (D2)}
fails
in this setting.  This is because
when $\alpha>2$, the behavior of the limiting process for the scaled processes is
not determined merely by
the expectation of random coefficients $\kappa(x,z)$
as in the the case of NNBRWs in random environments
(see, for instance,  \cite{BD}).

\bigskip

{\small
\noindent \textbf{Acknowledgements.} \rm
We thank the anonymous referee for valuable comments.
 The research of Xin Chen is supported by the
National Natural Science Foundation of China (Nos.\ 11501361 and 11871338).\
The research of Zhen-Qing Chen is partially supported by Simons Foundation Grant 520542 and a
 Victor Klee Faculty Fellowship at UW. \
 The research of Takashi Kumagai is supported
by JSPS KAKENHI Grant Number JP17H01093 and by the Alexander von Humboldt Foundation.\
The research of Jian Wang is supported by the National
Natural Science Foundation of China (Nos.
11831014 and 12071076),
the Program for Probability and Statistics: Theory and Application (No.\ IRTL1704) and the Program for Innovative Research Team in Science and Technology in Fujian Province University (IRTSTFJ).

}

\vskip 0.3truein
{\small
{\bf Xin Chen:}
   Department of Mathematics, Shanghai Jiao Tong University, 200240 Shanghai, P.R. China. \texttt{chenxin217@sjtu.edu.cn}
	
\bigskip
	
{\bf Zhen-Qing Chen:}
   Department of Mathematics, University of Washington, Seattle,
WA 98195, USA. \newline \texttt{zqchen@uw.edu}

\bigskip

{\bf Takashi Kumagai:}
 Research Institute for Mathematical Sciences,
Kyoto University, Kyoto 606-8502, Japan.
\texttt{kumagai@kurims.kyoto-u.ac.jp}

\bigskip

{\bf Jian Wang:}
    College of Mathematics and Informatics \& Fujian Key Laboratory of Mathematical
Analysis and Applications (FJKLMAA) \&
Center for Applied Mathematics of Fujian Province (FJNU), Fujian Normal University, 350007 Fuzhou, P.R. China.
     \texttt{jianwang@fjnu.edu.cn}}

\end{document}